\documentclass[a4paper,10pt]{article}
\usepackage[utf8]{inputenc}
\usepackage{amsmath}
\usepackage{amsthm}
\usepackage{amssymb}
\usepackage{hyperref}
\usepackage{pdfsync}
\usepackage{mathrsfs}
\usepackage{textcomp}
\usepackage{xcolor}
\usepackage{ulem}
\usepackage{bm}
\usepackage{authblk}

\DeclareMathOperator{\atan}{atan}

\newtheorem{thm}{Theorem}[section]
\newtheorem{prop}[thm]{Proposition}
\newtheorem{cor}[thm]{Corollary}
\newtheorem{lem}[thm]{Lemma}
\newtheorem{defn}[thm]{Definition}
\numberwithin{equation}{section}

\theoremstyle{remark}
\newtheorem{rmk}{Remark}[section]

\title{Input-to-State Stability in sup norms for hyperbolic systems with boundary disturbances}
\author[d]{Georges Bastin}
\author[c,e]{Jean-Michel Coron}
\author[a,b,e]{Amaury Hayat}
\affil[a]{CERMICS, Ecole des Ponts ParisTech, 6-8 Avenue Blaise Pascal, Champs-sur-Marne, France}
\affil[b]{Department of Mathematical Sciences and Center for Computational and Integrative Biology, Rutgers University–Camden, 303 Cooper St, Camden, NJ, USA}
\affil[c]{Sorbonne Université, Université de Paris, CNRS, Laboratoire Jacques-Louis Lions, Paris, France.}
\affil[d]{Department of Mathematical Engineering, ICTEAM, UCLouvain,
Louvain-La-Neuve, Belgium.}
\affil[e]{CAGE, INRIA, Paris, France.}

\date{\empty}

\begin{document}

\maketitle

\begin{abstract}
We give sufficient conditions for Input-to-State Stability in $C^{1}$ norm of general quasilinear hyperbolic systems with boundary input disturbances. In particular the derivation of explicit Input-to-State Stability conditions is discussed for the special case of $2\times 2$ systems.\\

\end{abstract}
\section{Introduction}
Hyperbolic systems are found everywhere in physical systems and sciences. From fluid dynamics to electromagnetism, cell growth, traffic transport, their ability to model propagation phenomena made them an unavoidable tool of many applications and led to hundreds of studies in the past decades. In most applications, one-dimensional quasilinear hyperbolic systems, around any steady state, can be written in the following form \cite{LiYu,BastinCoron1D,C1}:
\begin{equation}
\partial_{t}\mathbf{u}+A(\mathbf{u},x)\partial_{x}\mathbf{u}+B(\mathbf{u},x)=0, \;\;\; t \in [0, +\infty), \;\; x \in [0,L],
\label{sys1}
\end{equation}
\begin{equation}
\begin{pmatrix}
\mathbf{u}_{+}(t,0)\\
\mathbf{u}_{-}(t,L)
\end{pmatrix}=G\begin{pmatrix}
\mathbf{u}_{+}(t,L)\\
\mathbf{u}_{-}(t,0)
\end{pmatrix},
\label{bound}
\end{equation}
where 
\begin{itemize}
\item[(a)] $\mathbf{u}~: [0, +\infty) \times [0,L] \rightarrow \mathbb{R}^{n}$, 
\item[(b)] the maps $A$, $B$ and $G$ are $C^{1}$ and such that $A(0,x)=\Lambda(x)$ is diagonal, $B(0,x)=0$ and $G(0)=0$,
\item[(c)] the diagonal entries of $\Lambda(x)$ are denoted $\Lambda_i(x)$ and there exists $m \in \{1, \dots ,n\} $ such that, $\forall x \in [0,L]$, $\Lambda_{i}(x)>0$ for $i = 1,...,m$ and $\Lambda_{i}(x)<0$ for $i = m+1,...,n$, \textcolor{black}{and for any $i\neq j$, $\Lambda_{i}(x)\neq \Lambda_{j}(x)$}.
\item[(d)] $\mathbf{u}_+ \in \mathbb{R}^{m}$ and $\mathbf{u}_- \in \mathbb{R}^{n-m}$ are defined such that $\mathbf{u}^{\text{T}} = (\mathbf{u}_+^{\text{T}}, \mathbf{u}_-^{\text{T}})$. \textcolor{black}{Hence, $\mathbf{u}_{+}$ represent the components with positive propagation speeds and $\mathbf{u}_{-}$ the components with negative propagation speeds.}
\item[(e)] $\mathbf{u}(t,x) \equiv 0$, $\forall t \in [0, +\infty)$,  $\forall x \in [0,L]$, is the considered steady state.
\end{itemize}
The issue of the exponential stability of this \textcolor{black}{system} has attracted much attention in the last decades. The first result in the sup norm goes probably back to \cite{Li1984} in 1984 where Li and Greenberg studied an homogeneous system where $B=0$ and $G'(0)$ is diagonal, $m=1$ and $n=2$. This result was then generalized by \cite{1994-Li-book,Qin,Zhao,JCC,2010-Li-Rao-Wang-DCDS,CoronC1} to any $n\in\mathbb{N}^{*}$, and any $G$, but still with $B=0$. Inhomogeneous systems, when $B \neq 0$, were first treated in \cite{bastin2011coron, BastinCoron1D} in the $H^{2}$ norm which is easier to deal with, and then treated in the sup norms in \cite{C1,C1_22}. A more detailed review about these results and the main difficulties at each step of the generalization can be found in \cite[Section 1.6.1]{theseAH}.\\[0.8em]
In the present paper we address a slightly more general stability issue, namely the \textit{Input-to-State Stability} (ISS) of the system \eqref{sys1} when it is forced by a bounded boundary disturbance such that the boundary conditions are   
\begin{equation}
\begin{pmatrix}
\mathbf{u}_{+}(t,0)\\
\mathbf{u}_{-}(t,L)
\end{pmatrix}=G\begin{pmatrix}
\mathbf{u}_{+}(t,L)\\
\mathbf{u}_{-}(t,0)
\end{pmatrix}+\mathbf{d}(t),
\label{bound1}
\end{equation}
where $\mathbf{d}(t)\in\mathbb{R}^{n}$ is the boundary disturbance at time $t$. In this case the ISS measures the resilience of the system stability with respect to this disturbance or, in other words, how strongly the exponential stability of the steady state is changed by adding this disturbance. 
A precise definition is given in Definition \ref{defISS} below where it can be seen that
this ISS notion is more general since it implies the exponential stability of the steady state when the disturbance vanishes. The converse is false and the exponential stability of a system does not always imply its ISS and the existence of a Lyapunov function for a given steady state does not guarantee the ISS either as explained in \cite[Section 1.5 (C)]{KarafyllisKrstic}.\\[0.8em]
A natural question therefore arises: can the exponential stability results we mentioned above for system \eqref{sys1}--\eqref{bound} be extended to Input-to-State Stability for system \eqref{sys1}-\eqref{bound1}? In this article we will show that the answer is yes for the most up to date results, providing at the same time an improvement to the known ISS results in the sup norm.\\[0.8em]
The notion of ISS was first introduced by Sontag in 1989 \cite{Sontag1989} for finite dimensional systems. It was then extended to time delay systems, and then generalized to PDEs (see \cite[Chapter 1]{KarafyllisKrstic} for more details). In \cite[Part I-Part II]{KarafyllisKrstic},
 for instance, the authors give sufficient conditions for the ISS of a semilinear parabolic PDE or a linear hyperbolic PDE in the $L^{p}$ norm for any $p\in\mathbb{N}^{*}\cup\{+\infty\}$, including therefore the sup norm. In \cite{DashkovskiyMironchenko} the authors study ISS-Lyapunov functions and apply them to the ISS of semilinear reaction-diffusion equations for the $L^{2}$ norm.
  In \cite{Prieur2011} the authors study a linear parabolic system for the $L^{2}$ norm. In \cite{Prieur2012} the authors study a linear hyperbolic system with time varying coefficients and disturbances in the dynamics and for the $L^{2}$ norm. In \cite{Dashkovskiy} the authors show an ISS property for the semilinear wave equations for the sup norm, as well as a partial ISS property for the $L^{2}$ norm.
 In \cite{AP2016} the authors link the ISS for a nonlinear system in the $H^{p}$ norm to the behavior of storage functional, in \cite{Krsticmonotone} the authors link the ISS with the ISS with respect only to constant disturbances for monotonic nonlinear systems (which include parabolic PDEs with boundary disturbances). In \cite{yorgisbanda}, the authors show that the exponential stability results in the $H^{2}$ norm given in \cite{BastinCoron1D} can be extended under the same condition to ISS results (the linear case for the $L^{2}$ norm was shown in \cite{PrieurFerrante}). A more detailed review about the genesis of ISS notions for PDEs and some variations about the notion of ISS in infinite dimensional systems can be found in \cite[Chapter 1]{KarafyllisKrstic}. Some link with stability properties can also be found in \cite{MironchenkoWirth,Mironchenko2019}.
Other results about ISS have been developed in particular cases\textcolor{black}{:} in \cite{SaintVenantPI} is shown an ISS property for the Saint-Venant equations ; in \cite{Lhachemi2019} the authors study the ISS of a linear reaction-diffusion equation with a delay on the control input and a PI controller, etc.
But, to our knowledge, no general result exists in the sup norm.
In practice, however, the sup norms ($L^{\infty}$ or $C^{q}$ norms) are natural norms as, in physical systems, boundary disturbances are more likely to be uniformly bounded than to have a bounded $L^{p}$  or $W^{q,p}$ norm with $(p,q)\in(\mathbb{N}^{*})^{2}$. And from a mathematical point of view the $C^{1}$ norm is also the most natural norm for classical solutions of a quasilinear hyperbolic system. This is the problem we are investigating in this article.
In our main result, Theorem \ref{th1}, we give sufficient conditions to get ISS of general quasilinear hyperbolic systems
for the $C^{q}$ norm ($\textcolor{black}{q}\geq 1$), or the $L^{\infty}$ norm when the system is linear. To our knowledge, this is the first such general existing ISS result in sup norms for such systems.\\[0.8em]
The second part of the paper is devoted to the particular case of $2\times2$ systems of the following form
\begin{equation}
\partial_{t}\begin{pmatrix}u_{1}(t,x)\\u_{2}(t,x)\end{pmatrix}+A(\mathbf{u},x)\partial_{x}\begin{pmatrix}u_{1}(t,x)\\u_{2}(t,x)\end{pmatrix}+B(\mathbf{u},x)=0
\label{sys22}
\end{equation}
\begin{equation}
\text{ where }A(0,x)=\begin{pmatrix}\Lambda_{1}(x)&0\\0&\Lambda_{2}(x)\end{pmatrix}
\text{ and }\partial_{\mathbf{u}}B(0,x)=\begin{pmatrix}0&a(x)\\b(x)&0\end{pmatrix}.
\end{equation} 
Any quasilinear hyperbolic $2\times 2$ system can reduced to the form \eqref{sys22} (see \cite{bastin2011coron,Krsticsurvey} for instance).
These systems are interesting both from a practical and mathematical point of view. From a practical point of view they cover numerous physical systems in many areas from fluid mechanics (Euler Isentropic, Saint-Venant equations, etc.) to traffic flows (\cite{AwRascle, GaravelloVilla, CGARZ}), etc. From a mathematical point of view they represent the basic example of a coupled system that cannot be reduced to a homogeneous system.
As already mentioned, the most general known ISS results for \textcolor{black}{hyperbolic} systems deal with $2\times 2$ systems which are in addition linear, \textcolor{black}{and where $\Lambda_{1}$ and $\Lambda_{2}$ are constants}, and the state of the art is given in \cite[Chapter 9]{KarafyllisKrstic}.
We will show in Proposition \ref{th2} 
that our conditions provide an improvement to the previous conditions  when the system has \textcolor{black}{constant source term, i.e. $\partial_{\mathbf{u}}B(0,x)$ is constant}; and are necessary and sufficient when the system is homogeneous.

\section{Main results}
 
We consider the system \textcolor{black}{\eqref{sys1}, \eqref{bound1}. As stated in Theorem \ref{wp} herafter, this} system is well posed in $C^{1}$ (resp. $C^{q}$ for $q \geq 1$) for sufficiently small initial conditions satisfying the first order \textcolor{black}{(resp. $q$-th order)} compatibility conditions associated to \eqref{bound1} (see \cite{C1} or \cite[\textcolor{black}{(4.137)(4.142)}]{BastinCoron1D} for a precise definition of the first order compatibility condition). \\[0.8em]
Throughout the paper, the $C^{q}$ norm is denoted $\lVert\cdot \rVert_{C^{q}}$ and 
defined as follows for a function $\bm{\psi}=(\psi_{1},...,\psi_{n})^{T}\in C^{q}([0,L];\mathbb{R}^{n})$,
\begin{equation}
\lVert \bm{\psi} \rVert_{C^q}=\sum\limits_{k=0}^{q}\sup\limits_{i\in\{1,...n\}}\lVert \psi_{i}^{(k)} \rVert_{L^{\infty}}.
\end{equation}
\textcolor{black}{Also, for a vector $x=(x_{i})_{i\in\{1,...,k\}}$, the sup norm is denoted by $| x |=\max_{i}|x_{i}|$.}
We have the following theorem (see \cite{Wang}).

\begin{thm}[Well-posedness] \label{wp}
For all $T>0$ there exist $C_{\textcolor{black}{1}}(T)>0$ and $\delta(T)>0$ such that, for every $\mathbf{d}\in C^{1}([0,T])$, $\mathbf{u_{0}}\in C^{1}([0,L];\mathbb{R}^{n})$ satisfying the first order compatibility conditions and such that $\lVert \mathbf{u}_{0} \rVert_{C^1}+\lVert \mathbf{d} \rVert_{C^{1}}\leq \delta(T)$,
the system \eqref{sys1}, \eqref{bound1}, with $A$ and $B$ of class
$C^{1}$,
has a unique solution on $[0,T]\times[0,L]$ with initial condition $\mathbf{u_{0}}$. Moreover one has:
\begin{equation}
\lVert \mathbf{u}(t,\cdot)\rVert_{C^{1}}\leq C_{1}(T)\left( \lVert \mathbf{u}(0,\cdot)\rVert_{C^{1}}+\sup\limits_{\tau\in[0,t]}(\lvert \mathbf{d}(\tau)\rvert)+\sup\limits_{\tau\in[0,t]}(\lvert \mathbf{d}'(\tau)\rvert)\right),\text{ }\forall t\in[0,T].
\label{estimate}
\end{equation}
\label{th0}
\end{thm}
\begin{rmk}
When the maps $A$, $B$ and $G$ are of class $C^q$, this theorem can be generalized to the $C^q$ norm for any $q \geq 1$, by considering the augmented system $(\mathbf{u},\partial_{t}\mathbf{u},...,\partial_{t}^{q-1}\mathbf{u})$. In this case the right-hand side of the estimate \ref{estimate} includes the  derivatives of $\mathbf{d}$ up to order $q$. Besides, when the system is semilinear, i.e. $A(\mathbf{u},x)=A(x)$, this theorem holds also for the $C^{0}$ norm for $\mathbf{u}$ and the $L^{\infty}$ norm for $\mathbf{d}$.
\end{rmk}
We now introduce the definition of Input-to-State Stability,
\begin{defn}
We say that a system of the form \eqref{sys1}, \eqref{bound1} is (strongly) Input-to-State Stable (or ISS) with fading memory for the $C^{q}$ norm if there exist positive constants $C_{1}>0$, $C_{2}>0$, $\gamma>0$, and $\delta>0$ such that, for any $T>0$, for any $\mathbf{u_{0}}\in C^{q}([0,L];\mathbb{R}^{n})$ and for any $\mathbf{d} \in C^{q}([0,T];\mathbb{R}^{n})$ satisfying the $q$-th order compatibility conditions,  $\lVert \mathbf{u}_{0} \rVert_{C^{q}}\leq \delta $ and $\lVert \mathbf{d}\rVert_{C^{q}}\leq\delta $,
\begin{equation}
\lVert \mathbf{u}(t,\cdot)\rVert_{C^{q}}\leq C_{1}e^{-\gamma t}\lVert \mathbf{u}_{0}\rVert_{C^{q}}
+C_{2}\left(\sum\limits_{k=0}^{q}\sup\limits_{\tau\in[0,t]}\left(e^{-\gamma(t-\tau)}|\mathbf{d}^{(k)}(\tau)| \right)\right),
\label{ISS}
\end{equation}
\label{defISS}
\end{defn}
Note the fading-memory factor $e^{-\gamma(t-\tau)}$ in the last term which makes our definition of ISS slightly more strict than the usual definitions. For weaker notions of ISS, one can look for instance at \cite{KarafyllisKrstic} or \cite{MironchenkoWirth}.\\[0.8em]
In this article our major contribution is to show that the sufficient conditions derived in \cite{CoronC1,C1} for the exponential stability of quasilinear hyperbolic systems can be extended to the (strong) Input-to-State Stability of these systems. For the sake of clarity, in the next subsection, we start with the special case of homogeneous systems for which $B =0$. The general case will be considered next.

\subsection{The homogeneous case}
Let us first study the special case of homogeneous systems in which $B=0$. In this case the system \eqref{sys1} becomes
\begin{equation} 
\partial_{t}\mathbf{u}+A(\mathbf{u},x)\partial_{x}\mathbf{u}=0, \;\;\; t \in [0, +\infty), \;\; x \in [0,L],
\label{sys10}
\end{equation}
with boundary conditions \eqref{bound1}.\\[0.8em]
We recall the definition of the function $\rho_k~: \mathcal{M}_{n} \rightarrow \mathbb{R}$ which was already considered in
\cite[(2.7)]{1994-Li-book}, \cite[(1.4)]{CoronC1}, and \cite[(1.18)]{CoronNguyen}, and which is intrinsically linked to the stability of homogeneous systems in $C^{q}$ \textcolor{black}{with boundary conditions of the form \eqref{bound}}:
\begin{equation}
\rho_k(K)=\inf\{\lVert \Delta K \Delta^{-1} \rVert_k~: \Delta\in \mathcal{D}_{n}^{+}\}
\label{defrho1}
\end{equation}
where $\mathcal{M}_{n}$ is the space of $n\times n$ real matrices, $\mathcal{D}_{n}^{+}$ is the space of diagonal matrices with strictly positive diagonal entries, and
\textcolor{black}{
\begin{equation}
 \label{nrmminfty}
\| M\|_k  = \max_{\|\mathbf{\xi} \|_k = 1} \|M \mathbf{\xi}
\|_k \;\;  \forall M \in \mathcal{M}_{n},\;\; k\in\mathbb{N}^{*}\cup\{+\infty\},
\end{equation}
with
\begin{equation}
\label{normxinfty}
\| \mathbf{\xi}\|_k=\left(\sum\limits_{i=0}^{n}\xi_{i}^{k}\right)^{1/k}\;\text{ for }k\in\mathbb{N}^{*},\;\;\| \mathbf{\xi}\|_\infty=\max\left\{|\xi_i|;\, i\in
\{1,\cdots,n\}\right\}, \;\;  \forall 
\mathbf{\xi}=(\xi_1,\ldots,\xi_n)^{T}
\in \mathbb{R}^n.
\end{equation}}
We have the following ISS theorem.
\begin{thm} \label{th10}
Let an homogeneous quasilinear hyperbolic system be of the form \eqref{sys10}, \eqref{bound1}, with $A$ and $G$ of class $C^{q}$, with $q\in\mathbb{N}^{*}$. If
\begin{equation}
\rho_{\infty}(G'(0))<1,
\label{condrho1}
\end{equation}
then the system is Input-to-State Stable for the $C^{q}$ norm.
\end{thm}
The proof of this theorem is given in Section \ref{sectionproofth10}.
\begin{rmk}[Computing the values of the ISS gains]
The gains $C_{1}$ and $C_{2}$ in the ISS estimate \eqref{ISS} obtained by Theorem \ref{th10} can be expressed explicitly as a function of any matrix $\Delta$ such that $\lVert \Delta G'(0)\Delta\rVert_{\infty}<1$ (which exists from Condition \eqref{condrho1})
and the system parameters (see \ref{rmkconstants}).
\end{rmk}
Simple extensions of Theorem \ref{th10} are given in the two following remarks.
\begin{rmk}[Particular case of semilinear systems]\label{rmk1}
If the system \eqref{sys10} is semilinear (i.e. $A(\mathbf{u},x)=A(x)$), then Theorem \ref{th10} also holds true for $q=0$.
This is shown in Appendix \ref{semilinear}.
\end{rmk}
\begin{rmk}[Internal disturbances] \label{rmk-intern-disturb}
One could \textcolor{black}{also} include an internal distributed disturbance $\mathbf{d}_{2}(t,x)\in C^{q}([0,T];C^{0}([0,L];\mathbb{R}^{n}))$. The system \eqref{sys10} then becomes
\begin{equation}
\partial_{t}\mathbf{u}+A(\mathbf{u},x)\partial_{x}\mathbf{u}=\mathbf{d}_{2}(t,x), \;\;\; t \in [0, +\infty), \;\; x \in [0,L],
\end{equation}
and the same result holds with an ISS estimate rewritten as
\textcolor{black}{
\begin{equation}
\begin{split}
\lVert &\mathbf{u}(t,\cdot)\rVert_{C^{q}}\leq C_{1}e^{-\gamma t}\lVert \mathbf{u}_{0}\rVert_{C^{q}}
+ C_{2}\left(\sum\limits_{k=0}^{q}\sup\limits_{\tau\in[0,t]}\left(e^{-\gamma(t-\tau)}|\mathbf{d}^{(k)}(\tau)|\right)\right)\\
&+ C_{3}\left(\sup\limits_{(\tau,x)\in[0,t]\times[0,L]}\left(e^{-\gamma(t-\tau)}|\partial_{t}^{q}\mathbf{d}_{2}(\tau,x)|\right)+\sum\limits_{k_{1}+k_{2}\leq q-1}\sup\limits_{(\tau,x)\in[0,t]\times[0,L]}\left(e^{-\gamma(t-\tau)}|\partial_{t}^{k_{1}}\partial_{x}^{k_{2}}\mathbf{d}_{2}(\tau,x)|\right)\right),
\end{split}
\label{ISSdist}
\end{equation}}
instead of \eqref{ISS}. A way to adapt the proof is given in Appendix \ref{internal}.
\end{rmk}
Before to consider the inhomogeneous case in the next subsection, it is still interesting to point out the two following methodological remarks.
\begin{rmk}
First let us note that Condition \eqref{condrho1} is exactly the same as the sufficient condition that was given in \cite{1994-Li-book, Qin, Zhao} and \cite{CoronC1} for the exponential stability (in $C^{q}$) of the unforced system \eqref{sys10}, \eqref{bound} (i.e. without disturbance). 
\textcolor{black}{In the references \cite{Qin,Zhao,1994-Li-book}, the result relies on a careful estimate of the solutions and their derivatives along characteristics which might be hard to adapt to the ISS case.
In constrast,} in the reference \cite{CoronC1}, the exponential stability relies on a Lyapunov function equivalent to a sup-norm. In Theorems \ref{th10} and \ref{th1}, we shall show that the same Lyapunov function can be used as a so-called ISS Lyapunov function to extend the ISS property to the system \eqref{sys10}, \eqref{bound1} (i.e. in presence of the boundary disturbance). However, it should be noted that this extension is not as straightforward as one might think. The reason is that, in the ISS framework, it will appear that we are not able to get the usual differential inequality of the standard Lyapunov theory
\begin{equation}
\frac{dV(\mathbf{u}(t,\cdot))}{dt}\leq -C V(\mathbf{u}(t,\cdot))+ \gamma (\sum\limits_{k=0}^{q}\sup\limits_{\tau\in[0,t]}|\mathbf{d}^{(k)}(\tau)|),
\label{usualLyap}
\end{equation}
where $V$ denotes the Lyapunov function, $C$ is a positive constant and $\gamma$ is a class $\mathcal{K}$ function. In Sections \ref{sectionISS-Lyap-func} and \ref{sectionproofth10} we will see how it is possible to adapt the analysis to nevertheless prove Theorem \ref{th10} and get an estimate of the form \eqref{ISS} (see in particular \eqref{dWp}--\eqref{estimdWp2}).
\end{rmk}

\begin{rmk}
It is also worth noting that \eqref{condrho1} is only a sufficient condition. It is hard to decide whether this condition could be necessary or not, and if not, what would be the necessary condition. In fact, even for the exponential stability of the unforced system (i.e. without disturbances), this question remains unsolved.
The difficulty comes from the fact that, for nonlinear systems like \eqref{sys10}, the stability conditions in different norms are not equivalent. 
To clarify this point,
let us define ISS for the $H^{2}$ norm as follows.
\end{rmk}
\begin{defn}
\label{defISSH2}
We say that a system of the form \eqref{sys10}, \eqref{bound1} is (strongly) ISS with fading memory for the $H^{2}$ norm if there exist positive constants $C_{1}>0$, $C_{2}>0$, $\gamma>0$, and $\delta>0$ such that, for any $T>0$, for any $\mathbf{u_{0}}\in H^{2}([0,L];\mathbb{R}^{n})$ and for any $\mathbf{d} \in C^{2}([0,T];\mathbb{R}^{n})$ satisfying first order compatibility conditions,  $\lVert \mathbf{u}_{0} \rVert_{H^{2}}\leq \delta $ and $\lVert \mathbf{d}\rVert_{C^{2}}\leq\delta $,
\begin{equation}
\lVert \mathbf{u}(t,\cdot)\rVert_{H^2} \leq C_{1}e^{-\gamma t}\lVert \mathbf{u}_{0}\rVert_{H^{2}}
+C_{2}\left(\sum\limits_{k=0}^{2}\sup\limits_{\tau\in[0,t]}\left(e^{-\gamma(t-\tau)}|\mathbf{d}^{(k)}(\tau)|\right)\right).
\label{ISSH2}
\end{equation}
\end{defn}
With this definition, we have the following sufficient condition for ISS in the $H^{2}$ norm.
\begin{prop}
Let an homogeneous quasilinear hyperbolic system be of the form \eqref{sys10}, \eqref{bound1}, with $A$ and $G$ of class $C^{2}$. If
\begin{equation}
\rho_{2}(G'(0))<1,
\label{condrho2} 
\end{equation}
then the system is ISS for the $H^{2}$ norm.
\end{prop}
This proposition can be easily proved with a standard quadratic Lyapunov function \textcolor{black}{similar to the one used in} \cite{2008-Coron-Bastin-Andrea-Novel-SICON} (see also \cite{yorgisbanda} \textcolor{black}{for linear systems in the $L^{2}$ norm}). \textcolor{black}{This case is easier than the one we are dealing with in this article as one then obtains a classical Lyapunov estimate of the form \eqref{usualLyap}.}\\[0.8em] The point here is that (see \cite[Proposition 4.7]{BastinCoron1D}) 
\begin{equation}
	\rho_{2}(G'(0))\leq \rho_{\infty}(G'(0))
\end{equation}
and, furthermore, that there are systems of the form \eqref{sys10}, \eqref{bound1} for which this inequality is strict. From \cite[Theorem 2]{CoronNguyen}, we even know that, \textcolor{black}{for any $\epsilon > 0$,} there are systems such that $\rho_{2}(G'(0)) < 1 < \rho_{\infty}(G'(0)) \leq (1 + \epsilon)$ which are ISS for the \textcolor{black}{$H^{2}$} norm but \textcolor{black}{are Input-to-State} unstable for the $C^1$ norm. 

\subsection{The inhomogeneous case}
Let us now consider the general case, where $B\neq0$. In other words the system is inhomogeneous and has a source term. From a stability point of view, this changes the problem a lot. Indeed, the source term can strongly couple the equations~: while for the homogeneous case the system can be diagonalized such that the equations of the linearized system are coupled through the boundary conditions only, this cannot be done anymore when $B\neq0$. For the exponential stability it was shown in \cite{C1} that the the stability conditions of the homogeneous case can be generalized to inhomogeneous systems when $B\neq 0$, but an additional internal condition appears on the length of the domain $[0,L]$ \textcolor{black}{or equivalently} the magnitude of the source term (see also \cite[Proposition 5.12, Theorem 6.6]{BastinCoron1D}). In this paper, we shall see that similar limitations appear for ISS.\\[0.8em]
Theorem \ref{th10} can be generalized as follows.
\begin{thm}
Let a quasilinear hyperbolic system be of the form \eqref{sys1} with $A$ and $B$ of class $C^{q}$, with $q\in\mathbb{N}^{*}$. Let us denote $M(x)=\partial_{\mathbf{u}}B(0,x)$. Let us assume that the system
\begin{equation}
\Lambda_{i}(x)f_{i}'(x)\leq -2\left(-M_{ii}(x)f_{i}(x) + \sum\limits_{k=1,k\neq i}^{n}\lvert M_{ik}(x)\rvert \frac{f_{i}^{3/2}(x)}{\sqrt{f_{k}(x)}} \right),
\label{01cond111}
\end{equation}
has a solution $(f_{1},...,f_{n}):[0,L] \rightarrow \mathbb{R}^{\textcolor{black}{n}}$  such that $f_{i}(x) >0$ for all $i\in[1,n]$ and all $x \in [0,L]$ and there exists a diagonal matrix $\Delta$ with positive coefficients such that
\begin{equation}
\lVert \Delta G'(0)\Delta^{-1}\rVert_{\infty}<\frac{\inf_{i}\left(\frac{f_{i}(l_i)}{\Delta_{i}^{2}}\right)}{\sup_{i}\left(\frac{f_{i}(L-l_{i})}{\Delta_{i}^{2}}\right)},
\label{01condauxbords}
\end{equation}
where
$l_{i}=L$ if $\Lambda_{i}>0$ and $l_{i}=0$ otherwise.
Then the system \eqref{sys1}, \eqref{bound1} is Input-to-State Stable for the $C^{q}$ norm.
\label{th1}
\end{thm}
A way to adapt the proof of Theorem \ref{th10} is given in Appendix~\ref{inhomogeneous}. Again, the gain of the ISS estimate \eqref{ISS} can be computed from $\Delta$, the values $f_{i}(L-l_{i})$ and the parameters of the system. Moreover,  Remarks \ref{rmk1} and \ref{rmk-intern-disturb} still apply. 

\section{Comparison with existing ISS results and $2\times2$ systems.}

\subsection{Comparison with Karafyllis-Krstic condition for linear $2\times2$ systems}

To our knowledge, there are no other results in the literature for ISS in the sup-norm for general quasilinear systems. In the particular case of $2\times 2$ linear systems, the best existing result is the following, obtained by Karafyllis and Krstic in \cite[Section 9.4]{KarafyllisKrstic}. Consider a linear system of the form
\begin{gather}
\partial_{t}\begin{pmatrix}u_{1}(t,x)\\u_{2}(t,x)\end{pmatrix}+\begin{pmatrix}\Lambda_{1}&0\\0&\Lambda_{2}\end{pmatrix}\partial_{x}\begin{pmatrix}u_{1}(t,x)\\u_{2}(t,x)\end{pmatrix}+\begin{pmatrix}0&a(x)\\ b(x)&0\end{pmatrix}\begin{pmatrix}u_{1}(t,x)\\u_{2}(t,x)\end{pmatrix}=0
\label{linear1}\\
\begin{pmatrix}
\mathbf{u}_{1}(t,0)\\
\mathbf{u}_{2}(t,\textcolor{black}{1})
\end{pmatrix}=\begin{pmatrix}0&k_{1}\\k_{2}&0\end{pmatrix}
\begin{pmatrix}
\mathbf{u}_{1}(t,\textcolor{black}{1})\\
\mathbf{u}_{2}(t,0)
\end{pmatrix}+\mathbf{d}(t),
\label{linearbound}
\end{gather}
where $a(x)$ and $b(x)$ are continuous functions in $C^{0}([0,1])$, $\Lambda_{1}>0$ and $\Lambda_{2}<0$ are constant speed propagations, $k_{1}$ and $k_{2}$ are  constant parameters, and $\mathbf{d}\in L^{\infty}(\mathbb{R}_{+})$ is the boundary disturbance. Karafyllis and Krstic showed, using a small-gain analysis, that if there exists $K>0$ such that
\begin{equation}
\begin{split}
&(\left|k_{1}\right|+\left|k_{2}\right|)\exp(-K)<1,\\
&\left(\sqrt{\frac{\exp(2K)-\exp{K}}{
|\Lambda_{2}|K}B}+\sqrt{\lvert k_{2}\rvert}\right)\left(\sqrt{\frac{1-\exp(-K)}{\Lambda_{1}K}A}+\sqrt{\left|k_{1}\right|}\right)<1,\\
\text{where }&A:=\max\limits_{0\leq z\leq1}\left|a(z)\exp(2Kz)\right|\text{ and }B:=\max\limits_{0\leq z\leq1}\left|b(z)\exp(-2Kz)\right|,
\label{cond1}
\end{split}
\end{equation}
then the system \eqref{linear1}--\eqref{linearbound} is ISS for the $C^{0}$ norm.
Note that it is assumed here without loss of generality that $L=1$.
In this setting, we can apply Theorem \ref{th1} to the system \eqref{linear1}--\eqref{linearbound} and, if we assume in addition that $a$ and $b$ are constants,  we can compare our conditions to
\eqref{cond1}.
We have the following preliminary lemma.
\begin{lem}
For the system \eqref{linear1}--\eqref{linearbound} with $a$ and $b$ constant,  the conditions \eqref{01cond111}--\eqref{01condauxbords} of Theorem \ref{th1} are respectively equivalent to
\begin{equation}
\begin{split}
\text{(interior condition)    }&\left(\frac{\pi}{2}-\sqrt{\left|\frac{a b}{\Lambda_{1}\Lambda_{2}}\right|}\right)\geq0,\\
\text{(boundary conditions)    }&|k_{1}| <\sqrt{\left|\frac{a\Lambda_{2}}{b\Lambda_{1}}\right|}\tan\left(\frac{\pi}{2}-\sqrt{\left|\frac{a b}{\Lambda_{1}\Lambda_{2}}\right|}\right),\\
&|k_{2}|< \left|\frac{b\Lambda_{1}}{a\Lambda_{2}}\right| \left(\tan\left(\atan\left(\sqrt{\left|\frac{b\Lambda_{1}}{a\Lambda_{2}}\right|}|k_{1}|\right)+\sqrt{\left|\frac{a b}{\Lambda_{1}\Lambda_{2}}\right|}\right)\right)^{-1}.
\end{split}
\label{condprop2}
\end{equation}
\label{prop2}
\end{lem}
This lemma can then be used to prove the following proposition.
\begin{prop}
Consider the system \eqref{linear1}--\eqref{linearbound} with $a$ and $b$ constant.
Suppose there exists $K>0$ such that \eqref{cond1} holds.
Then the conditions \eqref{condprop2} of Lemma \ref{prop2}, and consequently the two conditions \eqref{01cond111}--\eqref{01condauxbords} of Theorem \ref{th1},
are satisfied.
\label{th2}
\end{prop}
Lemma \ref{prop2} and Proposition \ref{th2}  are proved in Section \ref{secprop2}.\\[0.8em]
It is interesting to note that Proposition \ref{th2} is a \textit{strict} implication. The converse does not hold in general. In fact it only holds when
$a=b=0$, in which case both conditions are equivalent to $|k_{1}k_{2}|<1$ which is the optimal (i.e. necessary and sufficient) condition \cite[Section 2.2.1]{BastinCoron1D}. This is shown in Appendix \ref{Appconverse}.

\subsection{Limit length of \textcolor{black}{ISS}}

Another way to interpret these results is by looking at the limit length \textcolor{black}{under which we can guarantee ISS of a given system when the boundary conditions \eqref{bound1} correspond to a boundary control with a boundary disturbance. This length is
defined as the maximal length below which 
ISS holds with $G=0$, as stated in the following definition.}
\begin{defn}
Let a system be of the form \eqref{sys1}.
We call $maximal$ $length$ $of$ $ISS$ the length $L_{\max}>0$ such that for any $L\in (0,L_{\max})$, 
the system \eqref{sys1}, \eqref{bound1} with $G=0$ defined on $[0,L]$ is ISS.
\end{defn}
Our result gives a lower bound on this length $L_{\max}$ as follows.
\begin{cor}
Let a system be of the form \eqref{sys1}, \eqref{bound1} with $G=0$. 
The length $L_{\max}$ is strictly positive, possibly infinite.
\label{corl}
\end{cor}
This follows directly from Theorem \ref{th1}. In practice, \textcolor{black}{Theorem \ref{th1} allows obtain a numerical lower bound of $L_{\max}$,} as follows: for any constant $\textcolor{black}{C}>0$ let $L(C)\in (0,+\infty]$ be such that the maximal solution of
\begin{equation}
\label{sysfcondi-C}
\left\{\begin{split}
\Lambda_{i}(x)f_{i}'(x)&= -2\left(-M_{ii}(x)f_{i}(x) + \sum\limits_{k=1,k\neq i}^{n}\lvert M_{ik}(x)\rvert \frac{f_{i}^{3/2}(x)}{\sqrt{f_{k}(x)}} \right),\text{ } x\geq 0,\\
f_{i}(0)&=\textcolor{black}{C}\text{ if }1\leq i\leq m,\\
f_{i}(0)&=0\text{ if }m+1\leq i\leq n.
\end{split}\right.
\end{equation}
is defined on $[0,L(C))$. Then $L(C)$ is a nondecreasing function of $C>0$ and, for every $C>0$,   $L(\textcolor{black}{C})\leq L_{\max}\in (0,+\infty]$. \textcolor{black}{Therefore a lower bound of $L_{\max}$ can be estimated in practice by choosing $C>0$ large enough and by solving numerically system \eqref{sysfcondi-C} in order to estimate $L(C)$.}
\\[0.8em]
Finally, to show the added value of this result, note that in the case of a $2\times 2$ system,
Theorem \ref{th1} provides an estimate of the limit length of ISS at least as good as the existing known result given in \cite{KarafyllisKrstic} as stated in the following proposition.
\begin{prop}
Let $k_{1}=k_{2}=0$ and assume that the condition \eqref{cond1} holds. Then, the conditions \eqref{01cond111}--\eqref{01condauxbords} of Theorem \ref{th1} are satisfied.
\label{proplimit}
\end{prop}
\section{Definition of an ISS Lyapunov function for the $C^q$ norm} \label{sectionISS-Lyap-func}
We define an ISS Lyapunov function for the $C^{q}$ norm.
\begin{defn}
An ISS Lyapunov function for the $C^{q}$ norm and the system \eqref{sys1}--\eqref{bound1} is a functional $V$ on $C^{q}([0,L])$ such that
there exists $\delta$, $\gamma$, $c_{1}$, $c_{2}$, $C_{1}$ and $C_{2}$ positive constants such that for any $T>0$, for any solution $\mathbf{u}\in C^{q}([0,T]\times [0,L];\mathbb{R}^{n})$ to the system \eqref{sys1}--\eqref{bound1} with $\lVert \mathbf{u}(t,\cdot)\rVert_{C^{q}}\leq \delta$ and $\lVert \mathbf{d}\rVert_{C^{q}}\leq \delta$,
\begin{equation}
c_{1} \lVert \mathbf{u}(t,\cdot) \rVert_{C^{q}}\leq V(\mathbf{u}(t,\cdot))\leq c_{2}\lVert \mathbf{u}(t,\cdot) \rVert_{C^{q}},
\label{equiv}
\end{equation}
and
\begin{equation}
V(\mathbf{u}(t,\cdot))\leq C_{1}e^{-\gamma (t-s)}V(\mathbf{u}(s,\cdot))+C_{2}\left(\sum\limits_{k=0}^{q}\sup\limits_{\tau\in[s,t]}\left(e^{-\gamma(t-\tau)}|\mathbf{d}^{(k)}(\tau)|\right)\right).
\label{estimateV}
\end{equation}
\end{defn}
Obviously if there exists an ISS Lyapunov function for the $C^{q}$ norm the system is ISS for the $C^{q}$ norm.
\begin{rmk}[Noticeable difference with the usual approach to ISS Lyapunov stability analysis]
In the usual approach for ISS stability analysis with ISS Lyapunov functions
it is generally required to get a differential inequality of the form
\begin{equation}
\frac{dV(\mathbf{u}(t,\cdot))}{dt}\leq - C V(\mathbf{u}(t,\cdot))+\gamma (\sum\limits_{k=0}^{q}\sup\limits_{\tau\in[0,t]}|\mathbf{d}^{(k)}(\tau)|),
\label{estimateVuse}
\end{equation}
where $C>0$ is a positive constant and $\gamma$ is a class $\mathcal{K}$  function. This inequality is then used to obtain an ISS stability estimate.  It is remarkable that, with the Lyapunov function \eqref{defV} for the $C^{q}$ norm that we shall use in the proof of Theorem \ref{th10}, it will appear that we are unable to obtain a differential inequality of the form \eqref{estimateVuse} (see \eqref{dWp}--\eqref{estimdWp2} hereafter), but we can nevertheless get the ISS estimate of the form \eqref{estimateV}. Note that this feature does not occur when studying ISS for $H^{2}$ norms (see Definition \ref{ISSH2} and Proposition \ref{condrho2}) or even $H^{q}$ norms, where it is still possible to obtain the required differential inequality (e.g. \cite{yorgisbanda, PrieurFerrante}).
\label{rmkdiffeq}
\end{rmk}
\section{Proof of Theorem \ref{th10}} \label{sectionproofth10}
In this section we prove Theorem \ref{th10}.

\begin{proof}[Proof of Theorem \ref{th10}]
The proof is based on some of the computations provided in \cite{C1}. For more convenience we use here the same notations as in \cite{C1}.
We will deal with the exponential stability of the nonlinear system for the $C^{1}$ norm as it is the most difficult case. The exponential stability for the $L^{\infty}$ norm can be done exactly similarly and this will be detailed in the appendix \ref{semilinear}, while the extension of the proof to the $C^{q}$ norm simply follows by considering an augmented system (see Appendix \ref{Cpnorm}).\\[0.8em]
 The idea is first to approximate a basic $C^{1}$ Lyapunov function by a function equivalent to the $W^{1,p}$ norm, then to prove some estimate independent of $p$ on these functions and finally to let $p$ tend to infinity.
Let us consider a hyperbolic system of the form \eqref{sys10}, \eqref{bound1} with $A$ of class $C^{2}$, let $T>0$ and let $\mathbf{u}$ be a $C^{2}$ solution on $[0,T]\times[0,L]$ such that $\lVert\mathbf{u}(0,\cdot)\rVert_{C^{1}}\leq \varepsilon$, and $\lVert \mathbf{d} \rVert_{C^{1}}\leq \delta$ where $\varepsilon$ and $\delta$ are positive constants to be determined. From Theorem \ref{th0}, such a solution exists provided the initial condition $\mathbf{u}(0,\cdot)$ satisfies the $C^{1}$ compatibility conditions.
Here we assume a $C^{q+1}$ regularity for the computations but we will recover the result for $C^{q}$ functions by density later on. For any $p\in \mathbb{N}^{*}$ and any $\mu>0$, we define the following function:
\begin{equation}
W_{1,p}=\left(\int_{0}^{L}\sum\limits_{i=0}^{n}f_{i}^{p}e^{-2p \mu s_{i} x}u_{i}^{2p}(t,x)dx\right)^{1/2p}, \label{defW1p}
\end{equation}
where $f_{i}$ are positive constants to be chosen and $s_{i}=1$ for $i\in\{1,...,m\}$ and $s_{i}=-1$ for $i\in\{m+1,...n\}$.
Similarly we define
\begin{gather}
W_{2,p}=\left(\int_{0}^{L}\sum\limits_{i=0}^{n}f_{i}^{p}e^{-2p \mu s_{i} x}(E(\mathbf{u},x)\partial_{t}\mathbf{u}(t,x))_{i}^{2p}dx\right)^{1/2p},\label{defW2p}\\
W_{p}=W_{1,p}+W_{2,p}.
\label{defWp}
\end{gather}
where $(\mathbf{u},x)\rightarrow E(\mathbf{u},x)$ is $C^{1}$ and $E(\mathbf{u},x)$ is a matrix diagonalizing $A(\mathbf{u},x)$\footnote{this function and matrix always exists from the strict hyperbolicity of $A$, provided $\lVert\mathbf{u}\rVert_{C^{0}}$ is small enough, see \cite[Lemma 6.7]{BastinCoron1D}}, and
where, by a slight abuse of notation, we use the compact notation $\partial_{t}\mathbf{u}$ to denote the function of the single variable $x$ defined as:
\begin{equation}
\partial_{t}\mathbf{u}:=-A(\mathbf{u},x)\partial_{x}\mathbf{u}.
\label{defdt}
\end{equation}
Note that, for a function $\mathbf{u}\in C^{1}([0,L])$ of one variable, this definition is consistent with the value of $\partial_{t}\mathbf{u}$ for any solution of \eqref{defdt}. Therefore $W_{p}$ can be defined for $C^{1}$ functions of one variable and depends on time only indirectly through $\mathbf{u}$. This justifies the slight abuse of notation \eqref{defdt}.
Clearly, $W_{p}$ is an approximation of the following $C^{1}$ Lyapunov function candidate
\begin{equation}
V(\mathbf{u})=\lVert\sqrt{f_{1}}e^{-\mu s_{i} \cdot}u_{1},...,\sqrt{f_{n}}e^{-\mu s_{i} \cdot}u_{n}\rVert_{L^{\infty}}+\lVert \sqrt{f_{1}}e^{-\mu s_{i} \cdot}\textcolor{black}{(E\partial_{t}\mathbf{u})_{1}},...,\sqrt{f_{n}}e^{-\mu s_{i} \cdot}\textcolor{black}{(E\partial_{t}\mathbf{u})_{n}}\rVert_{L^{\infty}}.
\label{defV}
\end{equation}
Indeed, as long as $\mathbf{u}$ is $C^{1}$, $W_{p}\rightarrow V$ when $p\rightarrow +\infty$. In order to show that $V$ is an ISS Lyapunov function, we will first study $W_{p}$ and show some ISS estimate which has a limit when $p\rightarrow+\infty$ and then deduce an estimate on $V$. For simplicity in the following, \textcolor{black}{we will sometimes drop the dependance in $t$.} 
Let us start by studying $W_{1,p}$. Differentiating $W_{1,p}$ with respect to $t$ along the $C^{2}$ solutions of \eqref{sys10}, \eqref{bound1}, from
 \cite[(3.23)-(3.29), (3.42)-(3.45)]{CoronC1} (see also \cite[(5.19)-(5.30)]{C1}),
there exist $\varepsilon_{1}>0$, $p_{1}\in\mathbb{N}^{*}$ and $\alpha_{0}>0$ a constant independent of $\mathbf{u}$ and $p$  such that for any $p\geq p_{1}$ and any positive $\varepsilon \leq \varepsilon_{1}$
\begin{equation}
\frac{d W_{1,p}}{dt}\leq-\textcolor{black}{I_{2}}
-\frac{\mu\alpha_{0}}{2}W_{1,p}+CW_{1,p}\lVert\mathbf{u}\rVert_{C^{1}}
\label{dW1p}
\end{equation}
where $C>0$ is a positive constant independent of $p$ and $\mathbf{u}$,
\begin{equation}
I_{2}=\frac{W_{1,p}^{1-2p}}{2p}\left[\sum\limits_{i=1}^{n}\lambda_{i}(\mathbf{u},x)f_{i}^{p}u_{i}^{2p}e^{-2p\mu s_{i}x}\right]_{0}^{L}
\label{defI21}
\end{equation}
with
$\lambda_{i}(\mathbf{u},x), i\in\{1,...,n\},$ being the
eigenvalues of $A(\mathbf{u},x)$ such that $\lambda_i(0,x)=\Lambda_i(0), i\in\{1,...,n\}$.
Let us look at $I_{2}$ which is the only term where the boundary disturbance occurs. We know that under assumption \eqref{01condauxbords}, there exists a diagonal matrix $\Delta=\text{diag}(\Delta_{1},...,\Delta_{n})$ with positive components such that
\begin{equation}
\theta:=\lVert \Delta G'(0)\Delta^{-1}\rVert_{\infty}=\sum\limits_{j=1}^{n}\sum\limits_{i=1}^{n}|(G'(0))_{i,j}|\frac{\Delta_{i}}{\Delta_{j}}<\textcolor{black}{1}
\label{deftheta}
\end{equation}
where we recall that $l_{i}=L$ if $\Lambda_{i}>0$ and $l_{i}=0$ otherwise.
In fact this implies that there exists $\alpha>0$ such that
\begin{equation}
(1+\alpha)\theta < 1
\label{defalpha}
\end{equation}
Note that when knowing $G$ \textcolor{black}{and} $\Delta$
one can derive an explicit maximal bound on $\alpha$. Let us now define the vector $\mathbf{\xi}=(\xi_{1},...,\xi_{n})^{T}$ by
\begin{equation}
\xi_{i}=\left\{\begin{array}{l}
\Delta_{i} u_{i}(t,L)\text{ for }i\in[1,m],\\
\Delta_{i} u_{i}(t,0)\text{ for }i\in[m+1,n],
\end{array}\right.
\label{defxi}
\end{equation}
thus, the $\xi_{i}$ correspond to the outgoing information of the system. Note that from \eqref{bound1} $G$ is only a function of $\xi$. Therefore, to simplify the notations in the following computations we define the function $F=(F_{i})_{i\in\{1,...,n\}}$ such that
\begin{equation}
F(\xi)=G\begin{pmatrix}\mathbf{u}_{+}(L)\\
\mathbf{u}_{-}(0)\end{pmatrix}
\label{defF}
\end{equation}
 Using the boundary conditions \eqref{bound1}, \eqref{defI21} and \eqref{defxi}, $I_{2}$ becomes
\begin{equation}
\begin{split}
I_{2}=\frac{W_{1,p}^{1-2p}}{2p}&\left[
\sum\limits_{i=1}^{m}\lambda_{i}(\mathbf{u}(t,L),L)\frac{f_{i}^{p}}{\Delta_{i}^{2p}}\xi_{i}^{2p}e^{-2p\mu L}\right.\\
&+\sum\limits_{i=m+1}^{n}|\lambda_{i}(\mathbf{u}(t,0),0)|\frac{f_{i}^{p}}{\Delta_{i}^{2p}}\xi_{i}^{2p}\\
&-\sum\limits_{i=1}^{m}\lambda_{i}(\mathbf{u}(t,0),0)f_{i}^{p}\left(F_{i}(\mathbf{\xi})+d_{i}\right)^{2p}\\
&\left.-\sum\limits_{i=m+1}^{n}|\lambda_{i}(\mathbf{u}(t,L),L)|f_{i}^{p}\left(F_{i}(\mathbf{\xi})+d_{i}\right)^{2p}e^{2p\mu L}
\right]
\end{split}
\label{I22}
\end{equation}
where $\mathbf{d}=(d_{i})_{i\in\{1,..,n\}}$ is the boundary disturbance.
Now, as $\lambda_{i}$ are $C^{1}$ functions of $\mathbf{u}$, and using \eqref{bound1}, we have
\begin{equation}
\begin{split}
I_{2}\geq \frac{W_{1,p}^{1-2p}}{2p}&\left[
\sum\limits_{i=1}^{m}(\Lambda_{i}(L)+O(\xi))\frac{f_{i}^{p}}{\Delta_{i}^{2p}}\xi_{i}^{2p}e^{-2p\mu L}\right.\\
&+\sum\limits_{i=m+1}^{n}(|\Lambda_{i}(0)|+O(\xi))\frac{f_{i}^{p}}{\Delta_{i}^{2p}}\xi_{i}^{2p}\\
&-\sum\limits_{i=1}^{m}\left(\Lambda_{i}(0)+O\left(\sum\limits_{i=1}^{n}(|F_{i}(\xi)|+|d_{i}|)\right)\right)f_{i}^{p}\left(F_{i}(\mathbf{\xi})+d_{i}\right)^{2p}\\
&\left.-\sum\limits_{i=m+1}^{n}\left(|\Lambda_{i}(L)|+O(\sum\limits_{i=1}^{n}(|F_{i}(\xi)|+|d_{i}|))\right)f_{i}^{p}e^{2p\mu L}\left(F_{i}(\mathbf{\xi})+d_{i}\right)^{2p}
\right]
\end{split}
\label{l23}
\end{equation}
Here the $O$ \textcolor{black}{represents a continuous function independent of $p$ such that $O(x)/|x|$ is bounded when $|x|$ tends to $0$}. As we have a bound on $\mathbf{u}$, hence on $F_{i}(\mathbf{\xi})$, and a bound on $d_{i}$,
a natural \textcolor{black}{idea would be} to develop $\left(F_{i}(\mathbf{\xi})+d_{i}\right)^{2p}$ and bound each of the terms that appear. However, this would pose a problem when making $p$ \textcolor{black}{tends} to $+\infty$ as we would end up with an infinite number of small terms and their sum might not be small anymore.
Note that this problem does not occur if one want to transpose the same type of result for the $H^{q}$ norm, where the Lyapunov function candidate would have a form similar to $\textcolor{black}{W}_{p}$ with $p=1$ \textcolor{black}{and would be $\sum_{k=1}^{q+1} W_{k,1}^2$, where $W_{k,p}$ is defined in the Appendix (see \eqref{Wkpappendix})}
and thus the number of \textcolor{black}{terms} would remain finite. In order to avoid this problem, \textcolor{black}{we introduce the following estimate for any $(a,d)\in \mathbb{R}^{2}$,
\begin{equation}
(a+d)^{2p}\leq \left(1+\alpha\right)^{2p}a^{2p}+\left(1+\frac{1}{\alpha}\right)^{2p}d^{2p}.
\label{estimad}
\end{equation}
Observe that this estimate holds for any positive $\alpha$ because we can} separate the possible situations in two cases, depending on which term is dominant between $a$ and $d$.
If $|a|\alpha\leq |d|$, then
\begin{equation}
(a+d)^{2p}\leq \left(1+\frac{1}{\alpha}\right)^{2p}d^{2p}.
\end{equation}
If $|a|\alpha> |d|$, then
\begin{equation}
(a+d)^{2p}\leq \left(1+\alpha\right)^{2p}a^{2p}.
\end{equation}
So overall \eqref{estimad} holds in any cases. For simplicity we also denote $d_{\max}(t)=\sup\limits_{i}|d_{i}(t)|$ and we recall the notation $l_{i}$ defined in Theorem \ref{th1} by $l_{i}:=L$ if $1\leq i\leq m$ and $l_{i}:=0$ if $m+1\leq i\leq n$. Therefore, using \eqref{estimad} in \eqref{l23}, 
 we get
\begin{equation}
\begin{split}
I_{2}\geq \frac{W_{1,p}^{1-2p}}{2p}&\left[
\sum\limits_{i=1}^{m}(\Lambda_{i}(L)+O(\xi))\frac{f_{i}^{p}}{\Delta_{i}^{2p}}\xi_{i}^{2p}e^{-2p\mu L}\right.\\
&+\sum\limits_{i=m+1}^{n}(|\Lambda_{i}(0)|+O(\xi))\frac{f_{i}^{p}}{\Delta_{i}^{2p}}\xi_{i}^{2p}\\
&-\sum\limits_{i=1}^{n}\left(|\Lambda_{i}(L-l_{i})|+O\left(|F(\xi)|+d_{\max}\right)\right)f_{i}^{p}e^{2p\mu (L-l_{i})}
\left(1+\alpha\right)^{2p}F_{i}^{2p}(\mathbf{\xi})\\
&\left.-\sum\limits_{i=1}^{n}\left(|\Lambda_{i}(L-l_{i})|+O\left(|F(\xi)|+d_{\max}\right)\right)f_{i}^{p}e^{2p\mu (L-l_{i})}\left(1+\frac{1}{\alpha}\right)^{2p}d_{\max}^{2p}
\right].
\end{split}
\label{I24}
\end{equation}
Now, we would like to deal with the negative terms. As we have separated
the influence of $\xi$ and $d_{\max}$, we want to compensate all negative terms in $F_{i}^{2p}(\xi)$ by 
using the two first positive terms of \eqref{I24} so that the sum is nonnegative. Note that the \textcolor{black}{cross} terms induced by the two last sums do not bring any difficulty as both $|\xi|$ (and therefore $|F(\xi)|$) and $d_{\max}$ can be made as small as desired by choosing $\varepsilon$ small enough (see \eqref{estimate} and \eqref{defxi}). Using \eqref{bound1}, \eqref{defF} and the fact that $G$ is $C^{1}$, we have, denoting $\textcolor{black}{J}:=G'(0)$,
\begin{equation}
F_{i}(\xi)=\sum\limits_{j=1}^{n}\textcolor{black}{J}_{i,j}\frac{\xi_{j}}{\Delta_{j}}+o\left(\xi\right),
\label{devlim}
\end{equation}
For a given $t\geq0$, there exists $i_{0}$ such that $\max_{i}(|\xi_{i}(t)|)=|\xi_{i_{0}}|$, thus, using the fact that $\sup\limits_{i}|F_{i}(\mathbf{\xi})|=O(\xi)$,
\begin{equation}
\begin{split}
&\sum\limits_{i=1}^{n}(|\Lambda_{i}(l_{i})|+O(\xi))\frac{f_{i}^{p}}{\Delta_{i}^{2p}}\xi_{i}^{2p}e^{-2p\mu l_{i}}
-\sum\limits_{i=1}^{n}\left(\Lambda_{i}(L-l_{i})+O(|\xi|+d_{\max})\right)\left(f_{i}^{p}e^{2p\mu (L-l_{i})}\right)(1+\alpha)^{2p}|F_{i}(\mathbf{\xi})|^{2p}\\
&\geq (\min\limits_{i}|\Lambda_{i}(l_{i})|+O(\xi))\min\limits_{i}\left(\frac{\textcolor{black}{f_{i}}}{\Delta_{i}^{2}}\right)^{p}\xi_{i_{0}}^{2p}e^{-2p\mu L}\\
&-n\textcolor{black}{\left(\sup\limits_{i}|\Lambda_{i}(L-l_{i})|+O(|\xi|+d_{\max})\right)}\sup\limits_{i}\left(\frac{\textcolor{black}{f_{i}}}{\Delta_{i}^{2}}\right)^{p}e^{2p\mu L}(1+\alpha)^{2p}
\left(\sum\limits_{j=1}^{n}|\textcolor{black}{J}_{i,j}|\frac{\Delta_{i}}{\Delta_{j}}\right)^{2p}\textcolor{black}{(1+o(1))^{2p}}|\mathbf{\xi}_{\textcolor{black}{i_{0}}}|^{2p}\\
&\geq
\left[(\min\limits_{i}|\Lambda_{i}(l_{i})|+O(\xi))\min\limits_{i}\left(\frac{\textcolor{black}{f_{i}}}{\Delta_{i}^{2}}\right)^{p}e^{-2p\mu L}\right.\\
&\left.-n\textcolor{black}{\left(\sup\limits_{i}|\Lambda_{i}(L-l_{i})|+O(|\xi|+d_{\max})\right)}\sup\limits_{i}\left(\frac{\textcolor{black}{f_{i}}}{\Delta_{i}^{2}}\right)^{p}e^{2p\mu L}(1+\alpha)^{2p}
\theta^{2p}\textcolor{black}{(1+o(1))^{2p}}\right]\mathbf{\xi}_{\textcolor{black}{i_{0}}}^{2p},
\label{I25b}
\end{split}
\end{equation}
\textcolor{black}{where $o(1)$ refers to a function that tends to $0$ when $|\xi|$ goes to $0$.}
By definition of $\alpha$, given in \eqref{defalpha}, there exists $p_{2}\in\mathbb{N}^{*}$ and $\mu_{2}>0$ such that for any $p\geq p_{2}$ and any $\mu\in(0,\mu_{2})$,
\begin{equation}
\begin{split}
&\textcolor{black}{n^{1/2p}\left(\sup\limits_{i}|\Lambda_{i}(L-l_{i})|+O(|\xi|+d_{\max})\right)^{1/2p}e^{\mu L}(1+\alpha)\theta}\\
&<e^{-4\mu L}(\min\limits_{i}|\Lambda_{i}(l_{i})|+O(\xi))^{1/2p}
\end{split}
\end{equation}
\textcolor{black}{Thus, from the definition of $o(1)$, using \eqref{estimate} and \eqref{defxi}, there exist $\varepsilon_{4}>0$ and $\delta_{1}>0$ such that for any $\varepsilon\in(0,\varepsilon_{4})$ and $\delta\in(0,\delta_{1})$,
\begin{equation}
\begin{split}
&\textcolor{black}{n^{1/2p}\left(\sup\limits_{i}|\Lambda_{i}(L-l_{i})|+O(|\xi|+d_{\max})\right)^{1/2p}e^{\mu L}(1+\alpha)(1+o(1))\theta}\\
&<e^{-4\mu L}(\min\limits_{i}|\Lambda_{i}(l_{i})|+O(\xi))^{1/2p}
\end{split}
\label{implyhomogene}
\end{equation}
}
\textcolor{black}{Thus, if we choose $f_{i}=\Delta_{i}^{2}$, \eqref{implyhomogene}}implies that
\begin{equation}
\begin{split}
&\left[(\min\limits_{i}|\Lambda_{i}(l_{i})|+O(\xi))\min\limits_{i}\left(\frac{\textcolor{black}{f_{i}}}{\Delta_{i}^{2}}\right)^{p}e^{-2p\mu L}\right.\\
&\left.-n\textcolor{black}{\left(\sup\limits_{i}|\Lambda_{i}(L-l_{i})|+O(|\xi|+d_{\max})\right)}\sup\limits_{i}\left(\frac{\textcolor{black}{f_{i}}}{\Delta_{i}^{2}}\right)^{p}e^{2p\mu L}(1+\alpha)^{2p}(1+o(1))^{2p}
\theta^{2p}\right]\mathbf{\xi}_{0}^{2p}>0,
\end{split}
\end{equation}
and hence, using \eqref{I25b},
\begin{equation}
\begin{split}
I_{2}\geq \frac{W_{1,p}^{1-2p}}{2p}\left[
-\left(\textcolor{black}{\sum\limits_{i=1}^{n}(|\Lambda_{i}(L-l_{i})|+O(|\xi|+d_{\max}))}\textcolor{black}{\Delta_{i}^{2p}}e^{2p\mu (L-l_{i})}\right)\left(1+\frac{1}{\alpha}\right)^{2p}d_{\max}^{2p}(t)
\right].
\end{split}
\end{equation}
This implies, using \eqref{dW1p}, that
\begin{equation}
\begin{split}
\frac{d W_{1,p}}{dt}\leq&-\frac{\mu\alpha_{0}}{2}W_{1,p}+CW_{1,p}\lVert\mathbf{u}\rVert_{C^{1}}\\
&+\frac{W_{1,p}^{1-2p}}{2p}\left[
\left(\textcolor{black}{\sum\limits_{i=1}^{n}|\Lambda_{i}(L-l_{i})|}\textcolor{black}{\Delta_{i}^{2p}}e^{2p\mu (L-l_{i})}\right)\left(1+\frac{1}{\alpha}\right)^{2p}d_{\max}^{2p}(t)\left(1+\textcolor{black}{O(\lVert \mathbf{u}(t,\cdot)\rVert_{C^{0}}+d_{\max}})\right)\right].
\end{split}
\label{dW1p2}
\end{equation}
Similarly we can show the following Lemma.
\begin{lem}
There exists $p_{3}\in\mathbb{N}^{*}$, $\mu_{3}>0$ such that for any $p\geq p_{3}$ and any $\mu\in(0,\mu_{3})$
\begin{equation}
\begin{split}
&\frac{d W_{2,p}}{dt}\leq-\frac{\mu\alpha_{0}}{2}W_{2,p}+CW_{2,p}\lVert\mathbf{u}\rVert_{C^{1}}\\
&+\frac{W_{2,p}^{1-2p}}{2p}\left[
\left(\textcolor{black}{\sum\limits_{i=1}^{n}|\Lambda_{i}(L-l_{i})|}\textcolor{black}{\Delta_{i}^{2p}}e^{2p\mu (L-l_{i})}\right)\left(1+\frac{1}{\alpha}\right)^{2p}(d'_{\max}(t))^{2p}\left(1+\textcolor{black}{O(\lVert \mathbf{u}(t,\cdot)\rVert_{C^{1}}+d_{\max}+d'_{\max}})\right)\right],
\end{split}
\end{equation}
where $d_{\max}'(t)=:\sup\limits_{i}|d_{i}'(t)|$.
\label{lemW2p}
\end{lem}
The proof is postponed to Appendix~\ref{appendixW2p}. Using \eqref{defWp}, \eqref{dW1p2} and Lemma \ref{lemW2p}, we have then
\begin{equation}
\begin{split}
\frac{dW_{p}}{dt}&\leq-\frac{\mu\alpha_{0}}{2}W_{p}+CW_{p}\lVert\mathbf{u}\rVert_{C^{1}}\\
&+\frac{W_{p}^{1-2p}}{2p}\left[
\left(\textcolor{black}{\sum\limits_{i=1}^{n}|\Lambda_{i}(L-l_{i})|}\textcolor{black}{\Delta^{2p}}e^{2p\mu (L-l_{i})}\right)\left(1+\frac{1}{\alpha}\right)^{2p}(d_{\max}^{2p}(t)\right.\\
&\left.+(d'_{\max}(t))^{2p})\left(1+\textcolor{black}{O(\lVert \mathbf{u}(t,\cdot)\rVert_{C^{1}}+d_{\max}+d'_{\max}})\right)
\right].
\end{split}
\end{equation}
\textcolor{black}{Using \eqref{estimate},} there exists $\varepsilon_{4}>0$ such that for any $\varepsilon\in(0,\varepsilon_{4})$ we have $C\lVert\mathbf{u}\rVert_{C^{1}}<\mu\alpha_{0}/8+CC_{1}(T)(\sup_{\tau\in[0,t]}(d_{\max}(\tau))+\sup_{\tau\in[0,t]}(d'_{\max}(\tau)))$. Besides, recall that $\lVert\mathbf{d}\rVert_{C^{1}}\leq \delta$ where $\delta>0$ has still to be chosen. Thus, we can select $\delta>0$ such that  $CC_{1}(T)\lVert\mathbf{d}\rVert_{C^{1}}\leq \mu\alpha_{0}/8$, thus $CC_{1}(T)(\sup_{\tau\in[0,t]}(d_{\max}(\tau))+\sup_{\tau\in[0,t]}(d'_{\max}(\tau)))\leq \mu\alpha_{0}/8$ and $C\lVert\mathbf{u}\rVert_{C^{1}}<\mu\alpha_{0}/4$.
Hence
\begin{equation}
\begin{split}
\frac{dW_{p}}{dt} &\leq-\frac{\mu\alpha_{0}}{4}W_{p}\\
&+\frac{W_{p}^{1-2p}}{2p}\left[
I_{4}\left(1+\frac{1}{\alpha}\right)^{2p}(d_{\max}^{2p}(t)+(d'_{\max}(t))^{2p})\left(1+\textcolor{black}{O(\lVert \mathbf{u}(t,\cdot)\rVert_{C^{1}}+d_{\max}+d'_{\max})}\right)\right],
\end{split}
\label{dWp}
\end{equation}
where we set 
 $I_{4}:=\left(\textcolor{black}{\sum\limits_{i=1}^{n}}\Lambda_{i}(L-l_{i})|\textcolor{black}{\Delta^{2p}}e^{2p\mu (L-l_{i})}\right)$.
\textcolor{black}{
Multiplying on both sides by $2pW_{p}^{2p-1}$, we get
\begin{equation}
\begin{split}
\frac{d(W_{p}^{2p})}{dt} &\leq-\frac{2p\mu\alpha_{0}}{4}W_{p}^{2p}\\
&+\left[
I_{4}\left(1+\frac{1}{\alpha}\right)^{2p}(d_{\max}^{2p}(t)+(d'_{\max}(t))^{2p})\left(1+O(\lVert \mathbf{u}(t,\cdot)\rVert_{C^{1}}+d_{\max}+d'_{\max})\right)\right].
\end{split}
\label{estimdWp2}
\end{equation}
Thus using now Gronwall Lemma, we get
\begin{equation}
\begin{split}
&W_{p}(\tau)\leq \left(e^{-2p\frac{\mu\alpha_{0}}{4} (\tau-s)}W_{p}(s)^{2p}\right.\\
&\left.+\int_{s}^{\tau}e^{-2p\frac{\mu\alpha_{0}}{4} (\tau-\textcolor{black}{v})}\left[I_{4}\left(1+\frac{1}{\alpha}\right)^{2p}(d_{\max}^{2p}(v)+(d'_{\max}(v))^{2p})\left(1+O(\lVert\mathbf{u}(v,\cdot)\rVert_{C^{1}}+d_{\max}(v)+d'_{\max}(v))\right)\right] dv\right)^{1/2p},\\
&\text{ for any }0\leq s\leq \tau \leq t.
\end{split}
\label{gronwall}
\end{equation}
Now, as $x\rightarrow x^{1/2p}$ is a concave function,
we have
\begin{equation}
\begin{split}
&W_{p}(t)\leq e^{-\frac{\mu\alpha_{0}}{4} (\textcolor{black}{t}-s)}W_{p}(s)\\
&+\left(\int_{s}^{t}e^{-2p\frac{\mu\alpha_{0}}{4} (t-\textcolor{black}{v})}\left[I_{4}\left(1+\frac{1}{\alpha}\right)^{2p}(d_{\max}^{2p}(v)+(d'_{\max}(v))^{2p})\left(1+O(\lVert\mathbf{u}(v,\cdot)\rVert_{C^{1}}+d_{\max}(v)+d'_{\max}(v))\right)\right] dv\right)^{1/2p}.
\end{split}
\label{gronwall2}
\end{equation}
We would like now to let $p$ go to $+\infty$ to recover the basic $C^{1}$ Lyapunov function $V$.
To do so, using the definition of $I_{4}$
together with the concavity of $1/2p$, there exists a constant $C_{3}$ independent of $p$ such that
\begin{equation}
\begin{split}
&\left(\int_{s}^{t}e^{-2p\frac{\mu\alpha_{0}}{4} (t-\textcolor{black}{v})}\left[I_{4}\left(1+\frac{1}{\alpha}\right)^{2p}(d_{\max}^{2p}(v)+(d'_{\max}(v))^{2p})\left(1+O(\lVert\mathbf{u}(v,\cdot)\rVert_{C^{1}}+d_{\max}(v)+d'_{\max}(v))\right)\right]\right)\\
&\leq C_{3}^{1/2p}\left(\int_{s}^{t}e^{-2p\frac{\mu\alpha_{0}}{4} (t-\textcolor{black}{v})}\left(\sum\limits_{i=1}^{n}\textcolor{black}{\Delta^{2p}}e^{2p\mu (L-l_{i})}\right)\left(1+\frac{1}{\alpha}\right)^{2p}(d_{\max}^{2p}(v)+(d'_{\max}(v))^{2p})dv\right)^{1/2p}.
\end{split}
\label{estimateI4I5p}
\end{equation}
Recall now that for a continuous function $a$, we have $(\int_{s}^{t}\sum_{i=1}^{n} |a_{i}|^{2p}(v)dv)^{1/2p}\xrightarrow[p\rightarrow +\infty]{}\max_{i, x\in[s,t]}\textcolor{black}{|a_{i}|}$,
therefore
\begin{equation}
\begin{split}
&C_{3}^{1/2p}\left(\int_{s}^{t}e^{-2p\frac{\mu\alpha_{0}}{4} (t-\textcolor{black}{v})}\left(\sum\limits_{i=1}^{n}\textcolor{black}{\Delta^{2p}}e^{2p\mu (L-l_{i})}\right)\left(1+\frac{1}{\alpha}\right)^{2p}d_{\max}^{2p}(v)dv\right)^{1/2p}\\
&\xrightarrow[p\rightarrow +\infty]{}\left(1+\frac{1}{\alpha}\right)\max\limits_{i\in\{1,...,n\}}\left(\textcolor{black}{\Delta_{i}}e^{\mu (L-l_{i})}\right)\sup\limits_{v\in[s,t]}\left( e^{-\frac{\mu\alpha_{0}}{4} (t-v)}|d_{\max}(v)|\right).
\end{split}
\label{boundI4p}
\end{equation}
and the same holds with $d_{\max}'$ instead of $d_{\max}$.
}
Using now \eqref{gronwall2} and \eqref{boundI4p}, we obtain by letting $p$ go to $+\infty$
\begin{equation}
\begin{split}
&V(t)\leq e^{-\gamma (t-s)}V(s)\\
&+\left(1+\frac{1}{\alpha}\right)\max\limits_{i\in\{1,...,n\}}\left(\textcolor{black}{\Delta_{i}}e^{\mu (L-l_{i})}\right)\left(\sup\limits_{\tau\in[s,t]}(\textcolor{black}{e^{-\gamma(t-\tau)}}|d_{\max}(\tau)|)+\sup\limits_{\tau\in[s,t]}(\textcolor{black}{e^{-\gamma(t-\tau)}}|d'_{\max}(\tau)|)\right),
\end{split}
\label{ISSth1}
\end{equation}
which is exactly the \textcolor{black}{desired ISS estimate}, \textcolor{black}{with $\gamma=\mu\alpha_{0}/4 $}.
We conclude by saying that $V$ is equivalent to the $C^{1}$ norm of $\mathbf{u}$ as, from \eqref{defV}, there exist positive constants $C_{\min}$ and $C_{\max}$ such that
\begin{equation}
C_{\min}\lVert\mathbf{u}\rVert_{C^{1}}\leq V\leq C_{\max}\lVert\mathbf{u}\rVert_{C^{1}},
\end{equation}
and $C_{\min}$ and $C_{\max}$ can be deduced explicitly from the $(f_{i}(L-l_{i}))_{i\in\{1,..,n\}}$ and the parameters of the system \eqref{sys1}. Therefore, we have
\begin{equation}
\begin{split}
&\lVert \mathbf{u}(t,\cdot)\rVert_{C^{1}}\leq \frac{C_{\max}}{C_{\min}}\lVert\mathbf{u}(s,\cdot)\rVert_{C^{1}}e^{-\textcolor{black}{\gamma} (t-s)}\\
&+\frac{1}{C_{\min}}\left(1+\frac{1}{\alpha}\right)\max\limits_{i\in\{1,...,n\}}\left(\textcolor{black}{\Delta_{i}}e^{\mu (L-l_{i})}\right)\left(\sup\limits_{\tau\in[s,t]}(\textcolor{black}{e^{-\textcolor{black}{\gamma}(t-\tau)}}|d_{\max}(\tau)|)+\sup\limits_{\tau\in[s,t]}(\textcolor{black}{e^{-\textcolor{black}{\gamma}(t-\tau)}}|d'_{\max}(\tau)|)\right).
\end{split}
\label{ISSth110}
\end{equation}
Finally, this estimate is true for solutions $\mathbf{u}\in C^{2}$  with $A$ which is a $C^{2}$ function, but it can be extended by density to solutions in $C^{1}$ with $A$ of class $C^{1}$ (see \cite{C1} or \cite[Lemma 4.2]{Burgers} where the same argument is detailed precisely with the $H^{2}$ norm).
\begin{rmk}\label{rmkconstants}
Note that from \eqref{defV} and the choice $f_{i}=\Delta_{i}^{2}$, the constants $C_{\min}$ and $C_{\max}$ can be directly computed from $\Delta$. And from \eqref{deftheta} and \eqref{defalpha}, $\alpha$ can be directly computed from $\Delta$ and $G'(0)$. Thus from \eqref{ISSth110} the gains of the ISS estimate can be explicitly computed from $\Delta$ and the system parameters.
\end{rmk}
\end{proof}
\section{Case $n=2$ and comparison with existing conditions}
In this section we prove Lemma \ref{prop2} and Proposition \ref{th2}. We first introduce a proposition which was shown in \cite[Theorem 3.2]{C1_22}\footnote{The conditions stated in \cite[Theorem 3.2]{C1_22} are in fact different than \eqref{01cond111}--\eqref{01condauxbords}, but are shown to be equivalent in the same paper (see \cite[Section 4]{C1_22})} \textcolor{black}{and} simplifies the condition of Theorem \ref{th1} in the case of a $2\times 2$ system.
\begin{prop}
Let a system be of the form \eqref{linear1}, \eqref{linearbound}, with $a$ and $b$ two continuous functions on $[0,1]$ and denote $M:=\begin{pmatrix}0&a\\b&0\end{pmatrix}$ and $G(\mathbf{u})=\begin{pmatrix}0&k_{1}\\k_{2}&0\end{pmatrix}\mathbf{u}$.
Then the two following are equivalent:
\begin{itemize}
\item Condition \eqref{01cond111}--\eqref{01condauxbords} holds.
\item There exists a solution $\eta$ on $[0,1]$ to
\begin{equation}
\left\{\begin{split}
\eta'&=\left|\frac{a}{\Lambda_{1}}\right|+\left|\frac{b}{|\Lambda_{2}|}\right|\eta^{2},\\
\eta(0)&=|k_{1}|
\end{split}\right.
\label{condn1}
\end{equation}
such that
\begin{equation}
\eta(1)\textcolor{black}{<} |k_{2}|^{-1}.
\label{condn2}
\end{equation}
\end{itemize}
\label{prop22}
\end{prop}
\subsection{Proof of Lemma \ref{prop2}}
\label{secprop2}
\textcolor{black}{
When $a$ and $b$ are constant, $\eta$ can be computed explicitly. Indeed, denoting $c_{1}=|a|/\Lambda_{1}$ and $c_{2}=|b|/|\Lambda_{2}|$, we have
\begin{equation}
\eta(x)=\sqrt{\frac{c_{1}}{c_{2}}}\tan(\atan(\sqrt{\frac{c_{2}}{c_{1}}}|k_{1}|)+\sqrt{c_{1}c_{2}}x),\text{ on $[0,x_{1})$},
\label{exprg}
\end{equation}
where $x_{1}$ is given by
\begin{equation}
x_{1}=\frac{\left(\pi/2-\atan(\sqrt{\frac{c_{2}}{c_{1}}}|k_{1}|)\right)}{\sqrt{c_{1}c_{2}}},
\end{equation}
and
\begin{equation}
\lim\limits_{x\rightarrow x_{1}} \eta(x) =+\infty.
\end{equation}
Therefore the existence of $\eta$ to \eqref{condn1} and \eqref{condn2} becomes
\begin{gather}
1<\frac{\left(\pi/2-\atan(\sqrt{\frac{c_{2}}{c_{1}}}|k_{1}|)\right)}{\sqrt{c_{1}c_{2}}}\\
k_{2}< \left(\sqrt{\frac{c_{1}}{c_{2}}}\tan(\atan(\sqrt{\frac{c_{2}}{c_{1}}}|k_{1}|)+\sqrt{c_{1}c_{2}})\right)^{-1}
\end{gather}
which is equivalent to \eqref{condprop2}. Together with Proposition \ref{prop22}, this ends the proof of Lemma \ref{prop2}.
}
\subsection{Proof of Proposition \ref{th2}}
In this subsection we prove \textcolor{black}{Proposition} \ref{th2}, by using Lemma \ref{prop2} and comparing the two conditions.
\begin{proof}[Proof of \textcolor{black}{Proposition} \ref{th2}]
\textcolor{black}{Assume that $a$ and $b$ are constant.} 
In this case, 
 $A=\textcolor{black}{|a|} e^{2K}$ and $B=\textcolor{black}{|b|}$. Thus, assuming that \eqref{cond1} holds, we have
\begin{equation}
\left(\sqrt{\frac{\exp(2K)-\exp(K)}{K}\frac{|b|}{|\Lambda_{2}|}}+\sqrt{|k_{2}|}\right)\left(\sqrt{\frac{\exp(2K)-\exp(K)}{K}\frac{|a|}{\Lambda_{1}}}+\sqrt{|k_{1}|}\right)<1,
\end{equation}
which implies in particular that
\begin{equation}
\left(\sqrt{\frac{|b|}{|\Lambda_{2}|}}+\sqrt{|k_{2}|}\right)\left(\sqrt{\frac{|a|}{\Lambda_{1}}}+\sqrt{|k_{1}|}\right)<1.
\label{cond20}
\end{equation}
\textcolor{black}{Denoting again} $c_{1}=|a|/\Lambda_{1}$ and $c_{2}=|b|/|\Lambda_{2}|$, we have
\begin{equation}
0\leq \sqrt{|k_{2}|}<\frac{1}{\sqrt{c_{1}}+\sqrt{|k_{1}|}}-\sqrt{c_{2}},
\label{estk2}
\end{equation}
and
\begin{equation}
0\leq \sqrt{|k_{1}|}<\frac{1}{\sqrt{c_{2}}+\sqrt{|k_{2}|}}-\sqrt{c_{1}}.
\label{estk1}
\end{equation}
From Lemma \ref{prop2}, it is enough to show that this implies that
\begin{equation}
x_{1}=\frac{\frac{\pi}{2}-\atan(\sqrt{\frac{c_{2}}{c_{1}}}|k_{1}|)}{\sqrt{c_{1}c_{2}}}>1,
\label{defetaunif}
\end{equation}
and that
\begin{equation}
k_{2}\leq \eta^{-1}(1),
\label{condk222}
\end{equation}
where $\eta(x)$ is defined again by 
\begin{equation}
\eta(x)=\sqrt{\frac{c_{1}}{c_{2}}}\tan(\atan(\sqrt{\frac{c_{2}}{c_{1}}}|k_{1}|)+\sqrt{c_{1}c_{2}}x),\text{ on $[0,x_{1})$}.
\end{equation}
\begin{itemize}
\item Proof that $x_{1}>1$. From \eqref{estk1},
\begin{equation}
\sqrt{\frac{\lvert k_{1}\rvert}{c_{1}}}<\left(\frac{1}{\sqrt{c_{1}c_{2}}}-1\right),
\label{condlongueur}
\end{equation}
which implies that
\begin{equation}
x_{1}=\frac{\left(\pi/2-\atan(\sqrt{\frac{c_{2}}{c_{1}}}|k_{1}|)\right)}{\sqrt{c_{1}c_{2}}}>\frac{\pi/2-\atan\left(\left(\frac{1}{(c_{1}c_{2})^{1/4}}-(c_{1}c_{2})^{1/4}\right)^{2}\right)}{\sqrt{c_{1}c_{2}}}.
\end{equation}
Note that from \eqref{condlongueur} $\sqrt{c_{1}c_{2}}<1$. Then, if $1>\sqrt{c_{1}c_{2}}\geq1/2$ then $x_{1}>\pi/2-\atan((2^{1/2}-1/2^{1/2})^2)>\pi/3>1$. If $1/3\leq\sqrt{c_{1}c_{2}}\leq1/2$ then $x_{1}>2(\pi/2-\atan((3^{1/2}-1/3^{1/2})^2))>2\pi/5>1$. It remains now only the case $\sqrt{c_{1}c_{2}}\leq1/3$. We use the two following facts for every $x>0$:
\begin{equation}
\begin{split}
\pi/2-\atan(x)&=\atan(1/x)\\
\atan(x)&\geq x-x^{3}/3
\end{split}
\end{equation}
Therefore we have
\begin{equation}
\begin{split}
x_{1}&>\frac{1}{\sqrt{c_{1}c_{2}}}\left(\frac{1}{((c_{1}c_{2})^{-1/4}-(c_{1}c_{2})^{1/4})^{2}}-\frac{1}{3((c_{1}c_{2})^{-1/4}-(c_{1}c_{2})^{1/4})^{6}}\right)\\
&=\left(\frac{1}{(1-(c_{1}c_{2})^{1/2})^{2}}-\frac{c_{1}c_{2}}{3(1-(c_{1}c_{2})^{1/2})^{6}}\right)
\end{split}
\end{equation}
Now, the function $y\rightarrow 1/(1-y)^{2}-y/(3(1-y)^{6})$ is strictly increasing then decreasing on $[0,1/3]$, thus one has
\begin{equation}
x_{1}>\max(1, 1/(1-1/3)^2-1/(9(1-1/3)^{6}))=1.
\end{equation}
Therefore in any cases $x_{1}>1$.

\item Proof that $|k_{2}|<\textcolor{black}{\eta^{-1}(1)}$\textcolor{black}{.}
As stated previously in \eqref{condk222} and using the definition of $g$ given by \eqref{defetaunif}, we only need to prove that
\begin{equation}
(\frac{1}{(\sqrt{c_{1}}+\sqrt{\lvert k_{1}\rvert})}-\sqrt{c_{2}})<\sqrt{\frac{\sqrt{\frac{c_{2}}{c_{1}}}}{\tan(\atan(\sqrt{\frac{c_{2}}{c_{1}}}|k_{1}|)+\sqrt{c_{1}c_{2}})}}.
\label{ck2}
\end{equation}
However, this could be rather tedious, thus we will look at an equivalent problem in order to bring ourselves in a similar setting as the proof that $x_{1}>1$. Suppose $|k_{2}|$ fixed and define $x_{2}\geq0$ such that
\begin{equation}
\textcolor{black}{\eta}(x_{2})=\frac{\tan(\atan(\sqrt{\frac{c_{2}}{c_{1}}}|k_{1}|)+\sqrt{c_{1}c_{2}}x_{2})}{\sqrt{\frac{c_{2}}{c_{1}}}}=|k_{2}|^{-1},
\label{g2}
\end{equation}
which is the limiting case for condition \eqref{condk222} to hold. Such $x_{2}$ exists and is positive as $\textcolor{black}{\eta}(0)=|k_{1}|<|k_{2}|^{-1}$ from \eqref{estk1},
and $\lim_{x\rightarrow x_{1}}\textcolor{black}{\eta}(x)=+\infty$.
Then we show that under the assumption \eqref{cond20}, we have $x_{2}>1$. As $g$ is strictly increasing this would give directly \eqref{condk222}, hence \eqref{condn2}. Thus it remains to show that $x_{2}>1$. From \eqref{g2} we have
\begin{equation}
\label{defx2}
x_{2}=\frac{\atan(\sqrt{\frac{c_{1}}{c_{2}}}|k_{2}|^{-1})-\atan(\sqrt{\frac{c_{2}}{c_{1}}}|k_{1}|)}{\sqrt{c_{1}c_{2}}}.
\end{equation}
From \eqref{estk1} we have
\begin{equation}
\sqrt{\sqrt{\frac{c_{2}}{c_{1}}}|k_{1}|}<\left(\frac{1}{(c_{2}c_{1})^{1/4}+\left(\frac{c_{1}}{c_{2}}\right)^{1/4}\sqrt{|k_{2}|}}-(c_{1}c_{2})^{1/4}\right).
\label{boundk1}
\end{equation}
Hence, using \eqref{boundk1} in \eqref{defx2}
\begin{equation}
x_{2}>\frac{\atan(\sqrt{\frac{c_{2}}{c_{1}}}|k_{2}|^{-1})-\atan\left(\left(\frac{1}{(c_{2}c_{1})^{1/4}+\left(\frac{c_{1}}{c_{2}}\right)^{1/4}\sqrt{|k_{2}|}}-(c_{1}c_{2})^{1/4}\right)^{2}\right)}{\sqrt{c_{1}c_{2}}}.
\label{defx22}
\end{equation}
We have an expression with a priori 3 parameters, $|k_{2}|$, $c_{1}$ and $c_{2}$. The first thing to realize is that, as previously, we can reduce it to 2 parameters by setting
$X:=\sqrt{c_{1}/c_{2}}|k_{2}|$ and $Y:=\sqrt{c_{1}c_{2}}$. Indeed, \eqref{defx22} becomes
\begin{equation}
x_{2}>\frac{\atan\left(\frac{1}{X}\right)-\atan\left(\left(\frac{1}{(Y)^{1/2}+\sqrt{X}}-Y^{1/2}\right)^{2}\right)}{Y},
\label{defx23}
\end{equation}
where $Y\in(0,1)$ and $X\in(0,(1/\sqrt{Y}-\sqrt{Y})^{2})$.
Observe however that we can in fact simplify again the expression with a new parametrization by setting $Z=\sqrt{X}+\sqrt{Y}$
such that $Z\in(\sqrt{Y},1/\sqrt{Y})$. Therefore \eqref{defx23} becomes
\begin{equation}
\begin{split}
x_{2}&>\frac{\atan\left(\frac{1}{(Z-\sqrt{Y})^{2}}\right)-\atan\left(\left(\frac{1}{Z}-Y^{1/2}\right)^{2}\right)}{Y}\\
&=\frac{\frac{\pi}{2}-\left[\atan\left((Z-\sqrt{Y})^{2}\right)+\atan\left(\left(\frac{1}{Z}-Y^{1/2}\right)^{2}\right)\right]}{Y}.
\end{split}
\label{defx24}
\end{equation}
We define the function $\phi:x\rightarrow \atan(a-x)-\atan(a)+x/(1+a^{2})$ where $a\geq x\geq0$. We have $\phi(0)=0$, and
\begin{equation}
\phi'(x)=-\frac{1}{1+(a-x)^{2}}+\frac{1}{1+a^{2}}\leq0,\text{  for  }x\in[0,a).
\end{equation}
Thus, $\phi(x)\leq 0$ on $[0,a]$ which implies that $\atan(a-x)\leq\atan(a)-x/(1+a^{2})$ for any $x\in[0,a]$. Using this in \eqref{defx24} gives
\begin{equation}
\begin{split}
x_{2}&>\frac{\frac{\pi}{2}-\left[\atan(Z)+\atan(Z^{-1})- \frac{(2\sqrt{Y}Z-Y)}{1+Z^{2}}-\frac{(2\sqrt{Y}Z^{-1}-Y)}{1+Z^{-2}}\right]}{Y}\\
&=\frac{\frac{\pi}{2}-\left[\frac{\pi}{2}- \frac{(2\sqrt{Y}Z-Y)}{1+Z^{2}}-\frac{(2\sqrt{Y}Z-YZ^{2})}{1+Z^{2}}\right]}{Y}\\
&=\frac{\left[\frac{(4\sqrt{Y}Z)}{1+Z^{2}}-Y\right]}{Y}\\
&=\left[\frac{4}{\sqrt{Y}}\left(\frac{Z}{1+Z^{2}}\right)-1\right].
\end{split}
\end{equation}
Then if we look at the function $l:x\rightarrow x/(1+x^{2})$, one has
\begin{equation}
l'(x)=\frac{1+x^{2}-2x^{2}}{(1+x^{2})^{2}}=\frac{1-x^{2}}{(1+x^{2})^{2}}.
\end{equation}
Thus, $l$ is increasing on $[0,1]$ and decreasing on $[1,+\infty)$ which implies that for any $Z\in(\sqrt{Y},1/\sqrt{Y})$
\begin{equation}
x_{2}>\frac{4}{\sqrt{Y}}\frac{\sqrt{Y}}{1+\sqrt{Y}^{2}}-1=\frac{4}{1+Y}-1.
\end{equation}
Thus, as $Y\in(0,1)$,
\begin{equation}
\begin{split}
x_{2}>1.
\end{split}
\label{endx2}
\end{equation}
\end{itemize}
This ends the proof of \textcolor{black}{Proposition} \ref{th2}.
\end{proof}
 \section*{Conclusion}
 In this paper we showed that the Lyapunov approach used to deal with the exponential stability of quasilinear hyperbolic systems can be adapted to the Input-to-State Stability in the $C^{q}$ norm, for any $q\geq 1$. The consequent sufficient conditions allow to derive explicit gains and lower bounds on the length of the interval such that ISS can be guaranteed. They also represent an improvement to the existing conditions for ISS of such systems.
 \section*{Acknowledgments}
 The authors would like to thank the ANR project Finite4SoS (No.ANR 15-CE23-0007), INRIA team CAGE, the NSF for support via the CPS Synergy project 
``Smoothing Traffic via Energy-efficient Autonomous Driving" (STEAD)
CNS 1837481 and the French Corps des IPEF for their financial support.
 \appendix
 \section{Proof of Theorem \ref{th1}}
 \label{inhomogeneous}
\textcolor{black}{When the system is inhomogeneous and $B\neq0$, one can define, similarly to \eqref{defW1p}--\eqref{defW2p}--\eqref{defWp},
\begin{gather}
W_{1,p}=\left(\int_{0}^{L}\sum\limits_{i=0}^{n}f_{i}^{p}(x)e^{-2p \mu s_{i} x}u_{i}^{2p}(t,x)dx\right)^{1/2p}\\
W_{2,p}=\left(\int_{0}^{L}\sum\limits_{i=0}^{n}f_{i}^{p}(x)e^{-2p \mu s_{i} x}(E(\mathbf{u},x)\partial_{t}\mathbf{u}(t,x))_{i}^{2p}dx\right)^{1/2p},\\
W_{p}=W_{1,p}+W_{2,p},
\label{defWpinhomo}
\end{gather}
where $f_{i}$ are now $C^{1}$ functions with value in $(0,+\infty)$ such that \textcolor{black}{\eqref{01cond111} and \eqref{01condauxbords} hold with strict inequalities}. The existence of such $f_i$'s follows from the assumptions of Theorem~\ref{th1} and the continuity of differential equations with respect to the right hand side (note that \eqref{01condauxbords} is a strict inequality). 
When differentiating $W_{1,p}$ along $C^{1}$ solutions, we get from \cite[(5.19)-(5.30)]{C1}, 
that there exist $\varepsilon_{1}>0$, $p_{1}\in\mathbb{N}^{*}$ and $\alpha_{0}>0$ a constant independent of $\mathbf{u}$ and $p$  such that for any $p\geq p_{1}$ and any positive $\varepsilon \leq \varepsilon_{1}$
\begin{equation}
\frac{d W_{1,p}}{dt}\leq-I_{2}-I_{3}-\frac{\mu\alpha_{0}}{2}W_{1,p}+CW_{1,p}\lVert\mathbf{u}\rVert_{C^{1}},
\label{dW1pinhomo}
\end{equation}
where $C>0$ is a positive constant independent of $p$ and $\mathbf{u}$, $I_{2}$ is still defined by \eqref{defI21} and $I_{3}$ is defined by
\begin{equation}
\begin{split}
I_{3}=&W_{1,p}^{1-2p}\int_{0}^{L}\sum\limits_{i=1}^{n}f_{i}^{p}(x)u_{i}^{2p-1}\left(\sum\limits_{k=1}^{n} M_{ik}u_{k}\right)e^{-2\mu s_{i}x}dx\\
&-\frac{W_{1,p}^{1-2p}}{2}\int_{0}^{L}\sum\limits_{i=1}^{n}\lambda_{i}(\mathbf{u},x)f_{i}^{p-1}(x)f_{i}'(x)u_{i}^{2p}e^{-2\mu s_{i}x}dx.
\end{split}
\end{equation}
From \textcolor{black}{\textcolor{black}{\cite[(5.35)-(5.38)]{C1}}, as \eqref{01cond111} holds with strict inequalities by assumption,
there exists $\varepsilon_{2}>0$} and $\mu_{1}>0$ such that \textcolor{black}{for any positive $\mu\leq \mu_{1}$, and any positive $\varepsilon\leq \varepsilon_{2}$,} $I_{3}>0$. Let us now look at $I_{2}$. The analysis has now to take into account that $f_{i}(L)\neq f_{i}(0)$ a priori. From \eqref{01condauxbords}, there exists $\alpha>0$ such that one has
\begin{equation}
(1+\alpha)\theta<\frac{\inf\limits_{i}\left(\frac{f_{i}(l_{i})}{\Delta^{2}_{i}}\right)}{\sup\limits_{i}\left(\frac{f_{i}(L-l_{i})}{\Delta^{2}_{i}}\right)}.
\label{defalphainhomo}
\end{equation}
And \eqref{I25b} becomes now
\begin{equation}
\begin{split}
&\sum\limits_{i=1}^{n}(|\Lambda_{i}(l_{i})|+O(\xi))\frac{f_{i}^{p}(l_{i})}{\Delta_{i}^{2p}}\xi_{i}^{2p}e^{-2p\mu l_{i}}
-\textcolor{black}{\sum\limits_{i=1}^{n}\left(\Lambda_{i}(L-l_{i})+O(|\xi|+d_{\max})\right)}\left(f_{i}^{p}(L-l_{i})e^{2p\mu (L-l_{i})}\right)(1+\alpha)^{2p}|F_{i}(\mathbf{\xi})|^{2p}\\
&\geq (\min\limits_{i}|\Lambda_{i}(l_{i})|+O(\xi))\min\limits_{i}\left(\frac{f_{i}(l_{i})}{\Delta_{i}^{2}}\right)^{p}\xi_{i_{0}}^{2p}e^{-2p\mu L}\\
&-n\textcolor{black}{\left(\sup\limits_{i}|\Lambda_{i}(L-l_{i})|+O(|\xi|+d_{\max})\right)}\sup\limits_{i}\left(\frac{f_{i}(L-l_{i})}{\Delta_{i}^{2}}\right)^{p}e^{2p\mu L}(1+\alpha)^{2p}
\left(\sum\limits_{j=1}^{n}|\textcolor{black}{J}_{i,j}|\frac{\Delta_{i}}{\Delta_{j}}\right)^{2p}\textcolor{black}{(1+o(1))^{2p}}|\mathbf{\xi}_{\textcolor{black}{i_{0}}}|^{2p}\\
&\geq
\left[(\min\limits_{i}|\Lambda_{i}(l_{i})|+O(\xi))\min\limits_{i}\left(\frac{f_{i}(l_{i})}{\Delta_{i}^{2}}\right)^{p}e^{-2p\mu L}\right.\\
&\left.-n\textcolor{black}{\left(\sup\limits_{i}|\Lambda_{i}(L-l_{i})|+O(|\xi|+d_{\max})\right)}\sup\limits_{i}\left(\frac{f_{i}(L-l_{i})}{\Delta_{i}^{2}}\right)^{p}e^{2p\mu L}(1+\alpha)^{2p}
\theta^{2p}\textcolor{black}{(1+o(1))^{2p}}\right]\mathbf{\xi}_{\textcolor{black}{i_{0}}}^{2p},
\label{I25binhomo}
\end{split}
\end{equation}
By definition of $\alpha$ and $o(1)$, there exist $\varepsilon_{4}>0$ and $\delta_{1}>0$ such that for any $\varepsilon\in(0,\varepsilon_{4})$ and $\delta\in(0,\delta_{1})$,
\begin{equation}
\begin{split}
&\textcolor{black}{n^{1/2p}\left(\sup\limits_{i}|\Lambda_{i}(L-l_{i})|+O(|\xi|+d_{\max})\right)^{1/2p}e^{\mu L}(1+\alpha)(1+o(1))\theta}\\
&<e^{-4\mu L}(\min\limits_{i}|\Lambda_{i}(l_{i})|+O(\xi))^{1/2p}\frac{\min\limits_{i}\left(\frac{f_{i}(l_{i})}{\Delta_{i}^{2}}\right)^{1/2}}{\sup\limits_{i}\left(\frac{f_{i}(L-l_{i})}{\Delta_{i}^{2}}\right)^{1/2}}.
\end{split}
\end{equation}
This implies that 
\begin{equation}
\begin{split}
I_{2}\geq \frac{W_{1,p}^{1-2p}}{2p}\left[
-\left(\textcolor{black}{\sum\limits_{i=1}^{n}(|\Lambda_{i}(L-l_{i})|+O(|\xi|+d_{\max}))}f_{i}^{p}(L-l_{i})e^{2p\mu (L-l_{i})}\right)\left(1+\frac{1}{\alpha}\right)^{2p}d_{\max}^{2p}(t)
\right].
\end{split}
\end{equation}
And recalling that 
$I_{3}>0$, we have
\begin{equation}
\begin{split}
\frac{d W_{1,p}}{dt}\leq&-\frac{\mu\alpha_{0}}{2}W_{1,p}+CW_{1,p}\lVert\mathbf{u}\rVert_{C^{1}}\\
&+\frac{W_{1,p}^{1-2p}}{2p}\left[
\left(\textcolor{black}{\sum\limits_{i=1}^{n}|\Lambda_{i}(L-l_{i})|}f_{i}^{p}(L-l_{i})e^{2p\mu (L-l_{i})}\right)\left(1+\frac{1}{\alpha}\right)^{2p}d_{\max}^{2p}(t)\left(1+\textcolor{black}{O(\lVert \mathbf{u}(t,\cdot)\rVert_{C^{0}}+d_{\max}})\right)\right].
\end{split}
\label{dW12pinhomo}
\end{equation}
The rest can be done similarly as previously to obtain the desired estimate 
\begin{equation}
\begin{split}
&\lVert \mathbf{u}(t,\cdot)\rVert_{C^{1}}\leq \frac{C_{\max}}{C_{\inf}}\lVert\mathbf{u}(s,\cdot)\rVert_{C^{1}}e^{-\mu (t-s)}\\
&+\frac{1}{C_{\inf}}\left(1+\frac{1}{\alpha}\right)\max\limits_{i\in\{1,...,n\}}\left(\sqrt{f_{i}(L-l_{i})}e^{\mu (L-l_{i})}\right)\left(\sup\limits_{\tau\in[s,t]}(\textcolor{black}{e^{-\mu(t-\tau)}}|d_{\max}(\tau)|)+\sup\limits_{\tau\in[s,t]}(\textcolor{black}{e^{-\mu(t-\tau)}}|d'_{\max}(\tau)|)\right).
\end{split}
\label{ISSth11}
\end{equation}
} 
\section{Adapting Theorem \ref{th1} for the $L^{\infty}$ norm in the semilinear case}
\label{semilinear}
If the system is semilinear, we just have to keep $W_{1,p}$ defined as previously, and ignore $W_{2,p}$ such that
\begin{equation}
W_{p}=W_{1,p}.
\end{equation}
When differentiating $W_{1,p}$ along $C^{2}$ solutions of \eqref{sys1}, \eqref{bound1}, we obtain this time
\begin{equation}
\frac{d W_{1,p}}{dt}\leq-I_{2}-I_{3}-\frac{\mu\alpha_{0}}{2}W_{1,p}+CW_{1,p}\lVert\mathbf{u}\rVert_{C^{0}}.
\end{equation}
The reason is that when differentiating once with respect to time along $C^{2}$ solutions, as $A(x)=\Lambda(x)$ and is diagonal,  we have
\begin{equation}
\frac{d W_{1,p}}{dt}=-W_{1,p}^{1-2p}\left(\int_{0}^{L}\sum\limits_{i=1}^{n}\Lambda_{i}f_{i}^{p}u_{i}^{2p-1}\partial_{x}u_{i}-\sum\limits_{i=1}^{n}\Lambda_{i}f_{i}^{p}u_{i}^{2p-1}\left(M(\mathbf{u},x)\mathbf{u}\right)_{i}\right),
\label{dW1linear}
\end{equation}
Thus the only nonlinear term is $M(\mathbf{u},x)\mathbf{u}\leq M(0,x)\mathbf{u}+C_{0}\lVert\mathbf{u}\rVert_{C^{0}}$, where $C_{0}$ is a constant depending only on the system, which explains that the only nonlinear corrections that appears in \eqref{dW1linear} involves $\lVert\mathbf{u}\rVert_{C^{0}}$ an not $\lVert\mathbf{u}\rVert_{C^{1}}$. Then one can deal with $I_{2}$ and $I_{3}$ exactly similarly as in the proof of \textcolor{black}{Theorem \ref{th10} and Theorem \ref{th1}} and we obtain instead of \eqref{dW1p2} \textcolor{black}{(resp. \eqref{dW12pinhomo} in the inhomogenous case)},
\begin{equation}
\begin{split}
\frac{d W_{1,p}}{dt}\leq&-\frac{\mu\alpha_{0}}{2}W_{1,p}+CW_{1,p}\lVert\mathbf{u}\rVert_{C^{0}}\\
&+\frac{W_{1,p}^{1-2p}}{2p}\left[
\left(\textcolor{black}{\sum\limits_{i=1}^{n}|\Lambda_{i}(L-l_{i})|}\textcolor{black}{f_{i}^{2p}(L-d_{i})}e^{2p\mu (L-l_{i})}\right)\left(1+\frac{1}{\alpha}\right)^{2p}d_{\max}^{2p}(t)\left(1+\textcolor{black}{O(\lVert \mathbf{u}(t,\cdot)\rVert_{C^{0}}+d_{\max}})\right)
\right].
\end{split}
\label{dW1p2linear}
\end{equation}
Now, assuming that \textcolor{black}{$\lVert\mathbf{u}_{0}\rVert_{C^{1}}\leq \varepsilon$, $\lVert \mathbf{d} \rVert_{C^{0}}\leq \varepsilon$, and from Theorem \ref{th0}, there exists $\varepsilon_{1}>0$ such that for any $\varepsilon\in (0,\varepsilon_{1})$
\begin{equation}
-\frac{\mu\alpha_{0}}{2}W_{1,p}+CW_{1,p}\lVert\mathbf{u}\rVert_{C^{0}}\leq -\frac{\mu\alpha_{0}}{4}W_{1,p}.
\end{equation}
The rest can be done identically as in \eqref{gronwall}--\eqref{ISSth1}.}
\section{Proof of Lemma \ref{lemW2p}}
\label{appendixW2p}
In this appendix we show how to adapt the proof of estimate \eqref{dW1p2} to obtain Lemma \ref{lemW2p}. \textcolor{black}{To avoid lengthening the article we prove it directly in the general case $B\neq0$.} As in \cite[(A.1)-(A.6) ]{C1}, by differentiating along the $C^{2}$ solutions of \eqref{sys1}, \eqref{bound1}, we get
\begin{equation}
\frac{dW_{2,p}}{dt}\leq -I_{21}-I_{31}-\left(\mu\alpha_{0}-\frac{C_{6}}{2p}\right)W_{2,p}+C_{5}W_{2,p}\lVert\mathbf{u}\rVert_{C^{1}},
\end{equation}
where $C_{5}$ and $C_{6}$ are constants that depend only on the system and
\begin{equation}
I_{21}=\frac{W_{2,p}^{1-2p}}{2p}\left[\sum\limits_{i=1}^{n}\lambda_{i}f_{i}^{p}(x)(E(\mathbf{u},x)\mathbf{u}_{t})_{i}^{2p}e^{-2p\mu s_{i}x}\right]_{0}^{L}
\end{equation}
and
\begin{equation}
\begin{split}
I_{31}=&W_{2,p}^{1-2p}\int_{0}^{L}\sum\limits_{i=1}^{n}f_{i}^{p}(x)(E\mathbf{u}_{t})_{i}^{2p-1}\left(\sum\limits_{k=1}^{n} R_{ik}(\mathbf{u},x)(E\mathbf{u}_{t})_{k}\right)e^{-2\mu s_{i}x}dx\\
&-\frac{W_{2,p}^{1-2p}}{2}\int_{0}^{L}\sum\limits_{i=1}^{n}\lambda_{i}(\mathbf{u},x)f_{i}^{p-1}(x)f_{i}'(x)(E\mathbf{u}_{t})_{i}^{2p}e^{-2\mu s_{i}x}dx,
\end{split}
\end{equation}
with $R=(R_{ij})_{(i,j)\in\{1,..,n\}}:=E(\mathbf{u},x)(D_{a}(\mathbf{u},x)+\frac{\partial B}{\partial\mathbf{u}})E^{-1}(\mathbf{u},x)$, and $D_{a}$ is the matrix with coefficients $\sum\limits_{k=1}^{n}\partial (A_{i,k}/\partial u_{j})(\mathbf{u}_{x})_{k}$. As previously in the proof of Theorem \ref{th1} dealing with $I_{31}$ can be done exactly as in \cite{C1} (see (A.7) to (A.9)). \textcolor{black}{Concerning $I_{21}$,} \textcolor{black}{from the definition of $E$, and the fact that $\lambda$ is $C^{0}$ we have 
\begin{equation}
\begin{split}
I_{21}=&\frac{W_{2,p}^{1-2p}}{2p}\left(\sum\limits_{i=1}^{n}(\Lambda_{i}+o(1))f_{i}^{p}(L)((\mathbf{u}_{t})_{i}(t,L)+o(|(\mathbf{u}_{t})(t,L)|))^{2p}e^{-2p\mu s_{i}L}\right.\\
&\left.-\sum\limits_{i=1}^{n}(\Lambda_{i}+o(1))f_{i}^{p}(0)((\mathbf{u}_{t})_{i}(t,0)+o(|(\mathbf{u}_{t})(t,0)|))^{2p}\right)
\end{split}
\end{equation}
where $o(x)$ refers to a function satisfying $o(x)/|x|\rightarrow 0$ when $\lVert\mathbf{u}\rVert_{C^{1}}$ tends to $0$. Then, differentiating \eqref{bound1}, we get
\begin{equation}
\begin{split}
\begin{pmatrix}
(\mathbf{u}_{t})_{+}(t,0)\\
(\mathbf{u}_{t})_{-}(t,L)
\end{pmatrix}&=G'\begin{pmatrix}
\mathbf{u}_{+}(t,L)\\
\mathbf{u}_{-}(t,0)
\end{pmatrix}\begin{pmatrix}
(\mathbf{u}_{t})_{+}(t,L)\\
(\mathbf{u}_{t})_{-}(t,0)
\end{pmatrix}+\mathbf{d'}(t)\\
&=(G'(0)+o(1))\begin{pmatrix}
(\mathbf{u}_{t})_{+}(t,L)\\
(\mathbf{u}_{t})_{-}(t,0)
\end{pmatrix}+\mathbf{d'}(t).
\label{bound12}
\end{split}
\end{equation}
Thus, defining $\mathbf{\xi}=(\xi_{1},...,\xi_{n})^{T}$ by
\begin{equation}
\xi_{i}=\left\{\begin{array}{l}
\Delta_{i} (\mathbf{u}_{t})_{i}(t,L)\text{ for }i\in[1,m],\\
\Delta_{i} (\mathbf{u}_{t})_{i}(t,0)\text{ for }i\in[m+1,n],
\end{array}\right.
\label{defxi2}
\end{equation}
and using \eqref{bound12} we have
\begin{equation}
\begin{split}
I_{21}=&\frac{W_{2,p}^{1-2p}}{2p}\left(\sum\limits_{i=1}^{m}(\Lambda_{i}(L)+o(1))\frac{f_{i}^{p}(L)}{\Delta_{i}}(\xi_{i}+o(|\xi|))^{2p}e^{-2p\mu s_{i}L}
+\sum\limits_{i=m+1}^{n}(|\Lambda_{i}(0)|+o(1))\frac{f_{i}^{p}(0)}{\Delta_{i}}(\xi_{i}(t,L)+o(|\xi|))^{2p}\right.\\
&\left.-\sum\limits_{i=1}^{n}(\Lambda_{i}(L-l_{i})+o(1))f_{i}^{p}(L-l_{i})e^{-2p\mu s_{i}(L-l_{i})}
(F_{2,i}(\xi)+o(|\xi|+d'_{\max}))^{2p}\right),
\end{split}
\end{equation}
where $d'_{\max}$ and $l_{i}$ are defined as previously and $F_{2}(\xi)$ is defined by
\begin{equation}
F_{2}(\xi)=G'\begin{pmatrix}
\mathbf{u}_{+}(t,L)\\
\mathbf{u}_{-}(t,0)
\end{pmatrix}\begin{pmatrix}
(\mathbf{u}_{t})_{+}(t,L)\\
(\mathbf{u}_{t})_{-}(t,0)
\end{pmatrix},
\end{equation}
as the right hand side is only a function of $\xi$ from \eqref{defxi2}. Therefore
\begin{equation}
F_{2,i}(\xi)=\sum\limits_{j=1}^{n}J_{i,j}\frac{\xi_{j}}{\Delta_{j}}+o(\xi),
\end{equation}
which is the analogous of \eqref{devlim}
}
\textcolor{black}{The rest can be done similarly as previously.}
\section{Extension of the proof of Theorems \ref{th10} and \ref{th1} to the $C^{q}$ norm}
\label{Cpnorm}
In order to extend the proof to the $C^{q}$ norm we consider the state $\mathbf{y}=(\mathbf{u},\partial_{t}\mathbf{u},...,\partial_{t}^{q-1}\mathbf{u})$. One can see that $\mathbf{y}$ is still solution of a quasilinear system of the form
\begin{equation}
\partial_{t}\mathbf{y}+\textcolor{black}{A_{1}}(\mathbf{y},x)\partial_{x}\mathbf{y}+\textcolor{black}{B_{1}}(\mathbf{y},x)=0,
\end{equation}
where $\textcolor{black}{A_{1}}$ is block diagonal as follows
\begin{equation}
A_{1} = \begin{pmatrix} A(\mathbf{u},x) & (0) & ...\\
(0) & A(\mathbf{u},x) & (0) & ...\\
(0)& (0) & A(\mathbf{u},x) & ...\\
... & ... & ... & ...
\end{pmatrix}
\end{equation}
and $M_{1}(x):=\partial_{y}B(0,x)$ is also block diagonal with blocks that are all $M(x)$. Thus we can define again
\begin{equation}
 W_{k+1,p}=\left(\int_{0}^{L} \sum\limits_{i=1}^{n}   f_{i}(x)^{p}e^{-2p \mu s_{i} x} (\textcolor{black}{E\partial_{t}^{k}\mathbf{u}})_{j}^{2p} e^{-2p\mu s_{i} x} dx \right)^{1/2p},
 \label{Wkpappendix}
 \end{equation}
 and consider $W_{p}=\sum\limits_{k=0}^{q} W_{k+1,p}$. The rest can be done is a similar way as previously.

\section{Adding internal disturbances}
\label{internal}
In this Appendix we show how to extend the results when there are internal disturbances as well in the system (see Remark \ref{rmk-intern-disturb}). \textcolor{black}{For simplicity, we deal with the homogeneous case when $B = 0$, even though the same could be done with the general inhomogeneous case}. If additional internal disturbances are included in the system , then, 
\textcolor{black}{the system becomes
\begin{equation}
\partial_{t}\mathbf{u}+A(\mathbf{u},x)\partial_{x}\mathbf{u} 
=\mathbf{d}_{2}(t,x), \;\;\; t \in [0, +\infty), \;\; x \in [0,L].
\label{sys1disturb}
\end{equation}
This implies a few changes in the Lyapunov stability analysis. For any $q\in\mathbb{N}^{*}$ we can define $W_{k+1,p}$ for $k\in\{1,...,q\}$ as in \eqref{Wkpappendix}. However, now, for a $C^{q+1}$ solution $\mathbf{u}$ to \eqref{sys1disturb} with boundary conditions \eqref{bound1}, an important difference occurs. One has
\begin{equation}
\begin{split}
\partial_{t}\mathbf{u}&=-A(\mathbf{u},x)\partial_{x}\mathbf{u}
+\mathbf{d}_{2}(t,x)\\
\partial_{t}^{2}\mathbf{u}&=A^{2}(\mathbf{u},x)\partial_{x}^{2}\mathbf{u}+A(\mathbf{u},x)\partial_{x}(A(\mathbf{u},x))\partial_{x}\mathbf{u}
+A(\mathbf{u},x)\partial_{x}\mathbf{d}_{2}(t,x)+\partial_{t}\mathbf{d}_{2}(t,x)\\
\partial_{t}^{k}\mathbf{u}&=(-1)^{k}A^{k}(\mathbf{u},x)\partial_{x}^{k}\mathbf{u}
+O\left(\left(\sum\limits_{i=0}^{k-1}|\partial_{x}^{i}\mathbf{u}|\right)^{2} +\sum\limits_{i_1+i_2\leq k-1}|\partial_{x}^{i_{1}}\partial_{t}^{i_{2}}\mathbf{d}_{2}(t,x)|\right),
\label{derivinternal}
\end{split}
\end{equation}
for $k\in\{1,...,q\}$, where $O(x)$ refers to a function such that $O(x)/|x|$ is bounded when $x\rightarrow 0$. 
Because of $\mathbf{d}_{2}$ and its derivatives, it could be that $\partial_{t}^{k}\mathbf{u}=0$ for any $k\in\{1,...,q\}$ while there exists $k\in\{1,...,q\}$ such that $\partial_{x}^{k}\mathbf{u}\neq 0$. Therefore, the Lyapunov function candidate we previously used, i.e. $V:=\lim_{p\rightarrow+\infty} \sum_{k=0}^{q}W_{k+1,p}$, is not equivalent anymore to the $C^{q}$ norm (recall that the $C^{q}$ norm is taken with respect to the $x$ derivatives and $W_{k+1,p}$ is given in \eqref{Wkpappendix}). To remedy this problem we define
\begin{equation}
V_{p}=\sum_{k=0}^{q}W_{k+1,p}
+\sum\limits_{k_1+k_2\leq q-1}\left(\int_{0}^{L}|\partial_{t}^{k_1}\partial_{x}^{k_2}\mathbf{d}_{2}(t,x)|^{2p}dx\right)^{1/2p}.
\label{defVpdist}
\end{equation}
In this case, our Lyapunov function candidate is now $V:=\lim\limits_{p\rightarrow +\infty} V_{p}$. Therefore, \textcolor{black}{from \eqref{Wkpappendix} and \eqref{derivinternal}, there exist $C_{\min}$ and $C_{\max}$ such that }
\begin{equation}
\begin{split}
&C_{\min}\left(\lVert\mathbf{u}\rVert_{C^{q}}
+\sum\limits_{k_1+k_2\leq q-1}\sup_{x\in[0,L]}\left|\partial_{t}^{k_1}\partial_{x}^{k_2}\mathbf{d}_{2}(t,x)\right|\right)\\
& \leq V\leq C_{\max} \left(
\lVert\mathbf{u}\rVert_{C^{q}}
+\sum\limits_{k_1+k_2\leq q-1}\sup_{x\in[0,L]}\left|\partial_{t}^{k_1}\partial_{x}^{k_2}\mathbf{d}_{2}(t,x)\right|\right).
 \label{equivdist}
 \end{split}
\end{equation}
It suffices now to obtain an ISS estimate on $W_{p}=\sum_{k=0}^{q}W_{k+1,p}$. Indeed, if there exist $p_{1}>0$ and $C>0$ independent of $\mathbf{u}$, $p$ and the disturbances such that for any $p\geq p_{1}$}
\textcolor{black}{\begin{equation}
W_{p}(t)\leq W_{p}(0)e^{-\gamma t}+C\sum\limits_{k=0}^{q}\left(\int_{0}^{t}e^{-2p\gamma(t-\tau)} |\mathbf{d}^{(k)}(\tau)|^{2p} d\tau\right)^{1/2p}+C\sum\limits_{k=0}^{q}\left(\int_{0}^{t}\int_{0}^{L}e^{-2p\gamma (t-\tau)}|\partial_{t}^{k}\mathbf{d}_{2}(\tau,x)|^{2p}dxdt\right)^{1/2p},
\label{estimateWpinternal}
\end{equation}
}
\textcolor{black}{
then, from \eqref{defVpdist},
\begin{equation}
\begin{split}
V_{p}(t)&\leq W_{p}(0)e^{-\gamma t}+C\sum\limits_{k=0}^{q} \left(\int_{0}^{t}e^{-2p\gamma(t-\tau)}|\mathbf{d}^{(k)}(\tau)|^{2p} d\tau\right)^{1/2p}+C\sum\limits_{k=0}^{q}\left(\int_{0}^{t}\int_{0}^{L}e^{-2p\gamma (t-\tau)}|\partial_{t}^{k}\mathbf{d}_{2}(\tau,x)|^{2p}dxdt\right)^{1/2p}\\
&
+\sum\limits_{k_1+k_2\leq q-1}\left(\int_{0}^{L}|\partial_{t}^{k_1}\partial_{x}^{k_2}\mathbf{d}_{2}(t,x)|^{2p}dx\right)^{1/2p},
\end{split}
\end{equation}
which, letting $p$ tend to $+\infty$, implies that
\begin{equation}
\begin{split}
V(t)\leq& V(0)e^{-\gamma t}
+C\sum\limits_{k=0}^{q}\sup\limits_{\tau\in[0,t]} |e^{-\gamma(t-\tau)}\mathbf{d}^{(k)}(\tau)|+C\sum\limits_{k=0}^{q}\sup_{(\tau,x)\in[0,t] \times [0,L]}|e^{-\gamma(t-\tau)}\partial_{t}^{k}\mathbf{d}_{2}(\tau,x)|\\
&
+\sum\limits_{k_1+k_2\leq q-1}\sup\limits_{x\in [0,L]}|\partial_{t}^{k_1}\partial_{x}^{k_2}\mathbf{d}_{2}(t,x)|dx,
\end{split}
\end{equation}
which implies, from \eqref{equivdist},
\begin{equation}
\begin{split}
C_{\min} \lVert \mathbf{u}\rVert_{C^{q}}\leq& C_{\max}
\lVert\mathbf{u}\rVert_{C^{q}}e^{-\gamma t}
+C\sum\limits_{k=0}^{q}\sup\limits_{\tau\in[0,t]} |e^{-2p\gamma(t-\tau)}\mathbf{d}^{(k)}(\tau)|\\
&+(C+C_{\max}+1)\left(\sup_{(\tau,x)\in[0,t] \times [0,L]}|e^{-\gamma(t-\tau)}\partial_{t}^{q}\mathbf{d}_{2}(\tau,x)|\right.\\
&\left.+\sum\limits_{k_1+k_2\leq q-1}\sup_{(\tau,x)\in[0,t] \times [0,L]}|e^{-\gamma(t-\tau)}\partial_{t}^{k_{1}}\partial_{x}^{k_{2}}\mathbf{d}_{2}(\tau,x)|\right),
\end{split}
\end{equation}
which gives} the desired ISS estimate \eqref{ISSdist}.
\textcolor{black}{It remains now only to proceed as previously for $W_{p}$ to obtain \eqref{estimateWpinternal}. When differentiating $W_{1,p}$ the only difference comes from the following additional term that} appears  in \eqref{dW1p},
\begin{equation}
I_{\textcolor{black}{5}}=-W_{1,p}^{1-2p}\int_{0}^{L}\sum\limits_{i=0}^{n}f_{i}^{p}(x)u_{i}^{2p-1}(t,x)d_{2,i}(t,x)dx,
\end{equation}
where $d_{2,i}(t,x)$ are internal disturbances. From there, using Young's inequality we get
\begin{equation}
I_{\textcolor{black}{5}}\geq W_{1,p}^{1-2p}\int_{0}^{L}\sum\limits_{i=0}^{n}f_{i}^{p}(x)\left(\frac{2p-1}{2p}u_{i}^{2p}(t,x)+\frac{1}{2p}d_{2,i}^{2p}(t,x)\right)dx=\frac{\mu\alpha_{0}}{8}\frac{2p-1}{2p}W_{1,p}
+\frac{W_{1,p}^{1-2p}}{2p}\left(\frac{8}{\mu\alpha_{0}}\right)^{2p-1}D_{1,p}^{2p},
\end{equation}
where 
\begin{equation}
D_{1,p}=\left(\int_{0}^{L}\sum\limits_{i=0}^{n}f_{i}^{p}d_{2,i}^{2p}(t,x)dx\right)^{1/2p}.
\end{equation}
and therefore \eqref{dWp} becomes
\begin{equation}
\begin{split}
\frac{dW_{p}}{dt} &\leq-\frac{\mu\alpha_{0}}{8}W_{p}\\
&+\frac{W_{p}^{1-2p}}{2p}\left[
I_{4}\left(1+\frac{1}{\alpha}\right)^{2p}(d_{\max}^{2p}(t)+(d'_{\max}(t))^{2p})\left(1+O(\lVert \mathbf{u}(t,\cdot)\rVert_{C^{1}})\right)\right.\\
&\left.+\left(\frac{8}{\mu\alpha_{0}}\right)^{2p-1}D_{1,p}^{2p})\right].
\end{split}
\end{equation}
The rest can be done similarly \textcolor{black}{to get \eqref{estimateWpinternal}}.
\section{Converse of \textcolor{black}{Proposition} \ref{th2} does not hold if $a\neq0$ or $b\neq0$}
In this section we show that \textcolor{black}{Proposition} \ref{th2} is a strict implication when $a\neq 0$ or $b\neq 0$. Let $a$ and $b$ be such that 
\begin{equation}
\left|\frac{ab}{\Lambda_{1}\Lambda_{2}}\right|\leq\frac{\pi}{2}.
\end{equation}
Let $k_{1}$ satisfy the first condition of \eqref{condprop2} , and let
$\varepsilon>0$ sufficiently small to be determined later on, and define
\begin{equation}
k_{2}=\eta^{-1}(1)-\varepsilon,
\end{equation}
where $\eta$ is given by \eqref{defetaunif}.
From Proposition \ref{prop22}, the conditions \eqref{01cond111}--\eqref{01condauxbords} of Theorem \ref{th1} are satisfied. We will now show that for $\varepsilon$ small enough, condition \eqref{cond1} is not satisfied. Let assume by contradiction that \eqref{cond1} is satisfied. Then $1/(\sqrt{c_{1}}+\sqrt{|k_{1}|})-\sqrt{c_{2}}> \sqrt{|k_{2}|}$, and by continuity
there exists $k_{0}>|k_{2}|$  such that
\begin{equation}
\sqrt{k_{0}}<\left(\frac{1}{\sqrt{c_{1}}+\sqrt{|k_{1}|}}-\sqrt{c_{2}}\right).
\label{defk0}
\end{equation}
Now,  we define $x_{0}$ such that $\eta(x_{0})=k_{0}^{-1}$, which is possible as $\eta(0)=k_{1}<k_{0}^{-1}$. From \eqref{defk0}, $\eta$ is strictly increasing and goes to $+\infty$ in finite time. Proceeding as previously in \eqref{g2}--\eqref{endx2} (note that $k_{0}$ here satisfies the same assumption as $k_{2}$ in the proof of Proposition \ref{th2}), we have $x_{0}>1$. As $x_{0}$ and $k_{0}$ do not depend on $\varepsilon$, and as $\eta$ is strictly increasing, we can choose $\varepsilon>0$ small enough such that 
\begin{equation}
k_{2}^{-1}=\frac{\eta(1)}{1-\varepsilon\eta(1)}<\eta(x_{0})= k_{0}^{-1}.
\end{equation}
Thus $k_{0}<k_{2}$. But, by definition, $k_{0}>k_{2}$ so we have a contradiction and \eqref{cond1} is not satisfied. 
\label{Appconverse}

\section{Proof of Proposition \ref{proplimit}}
In this section $k_{1}=k_{2}=0$ and we assume the existence of $K>0$ such that \eqref{cond1} holds. 
We will show that conditions \eqref{01cond111}--\eqref{01condauxbords} of Theorem \ref{th1} hold (for $k_{1}=k_{2}=0$). From \eqref{cond1} we have
\begin{equation}
AB\left(\frac{e^{K}-1}{K}\right)^{2}<|\Lambda_{1}\Lambda_{2}|.
\label{cond1l}
\end{equation}
\textcolor{black}{Define $\eta$ as the maximal solution of \eqref{condn1} with $\eta(0)=0$ and $\eta_{K}=\eta e^{2Kx}$. From Cauchy-Lipschitz Theorem $\eta$ is defined on $[0,x_{4})$ and $x_{4}=+\infty$ or $\lim_{x\rightarrow x_{4}}\eta(x)=+\infty$.
From Proposition \ref{prop22}, we only need to show that $x_{4}>1$.
Using \eqref{condn1}, we have
\begin{equation}
\begin{split}
\eta_{K}'&=\eta'e^{2Kx}+2K\eta e^{2Kx}=|\frac{a(x)}{\Lambda_{1}}e^{2Kx}|+|\frac{b(x)}{|\Lambda_{2}|}e^{-2Kx}|\eta_{K}^{2}+2K\eta_{K},\\
\eta_{K}(0)&=0.\\
\end{split}
\end{equation}
This was done to make  $|a(x)e^{2Kx}|$ and $|b(x)e^{-2Kx}|$ appear, whose maxima on $[0,L]$ are respectively given by $A$ and $B$. Thus if we define $h$ as the maximal solution of
\begin{equation}
\begin{split}
&h'=\textcolor{black}{A_{1}}+\textcolor{black}{B_{1}}h^{2}+2Kh,\\
&h(0)=0,
\end{split}
\label{defh}
\end{equation}
\textcolor{black}{where $\textcolor{black}{A_{1}}=A/\Lambda_{1}$ and $\textcolor{black}{B_{1}}= B/|\Lambda_{2}|$,} and $[0,x_{5})$ its maximal domain of definition,
by comparison \cite{Hartman} one has $0\leq \eta_{K}(x)\leq h(x)$ on $[0,x_{5})$ and in particular $x_{4}\geq x_{5}$. We will now show that $x_{5}>1$. If $B_{1}=0$, then $x_{5}=+\infty$ as the equation is linear, so we can restrict ourselves to the case $B_{1}> 0$. Equation \eqref{defh} can be solved and we have, if $\sqrt{A_{1}B_{1}}>K^{2}$, then
\begin{equation}
h(x)=\frac{\sqrt{\textcolor{black}{A_{1}}\textcolor{black}{B_{1}}-K^{2}}\tan\left(\atan\left(\frac{K}{\sqrt{\textcolor{black}{A_{1}}\textcolor{black}{B_{1}}-K^{2}}}\right)+x\sqrt{\textcolor{black}{A_{1}}\textcolor{black}{B_{1}}-K^{2}}\right)-K}{\textcolor{black}{B_{1}}}.
\end{equation}
and
\begin{equation}
x_{5}=\frac{1}{\sqrt{A_{1}B_{1}-K^{2}}}\left[\frac{\pi}{2}-\atan\left(\frac{K}{\sqrt{\textcolor{black}{A_{1}}\textcolor{black}{B_{1}}-K^{2}}}\right)\right]=
\frac{\atan\left(\frac{\sqrt{\textcolor{black}{A_{1}}\textcolor{black}{B_{1}}-K^{2}}}{K}\right)}{\sqrt{\textcolor{black}{A_{1}}\textcolor{black}{B_{1}}-K^{2}}}.
\end{equation}
If we look at the function $r:K\rightarrow \atan(\sqrt{A_{1}B_{1}-K^{2}}/K)-\sqrt{A_{1}B_{1}-K^{2}}$, we have
\begin{equation}
\begin{split}
r'(K)&=\frac{(\sqrt{A_{1}B_{1}-K^{2}})'K-\sqrt{A_{1}B_{1}-K^{2}}}{K^{2}+\left(\sqrt{A_{1}B_{1}-K^{2}}\right)^{2}}-(\sqrt{A_{1}B_{1}-K^{2}})'\\
&=\frac{(\sqrt{A_{1}B_{1}-K^{2}})'(K-K^{2})-\sqrt{A_{1}B_{1}-K^{2}}(1+(\sqrt{A_{1}B_{1}-K^{2}})'(\sqrt{A_{1}B_{1}-K^{2}}))}{K^{2}+\left(\sqrt{A_{1}B_{1}-K^{2}}\right)^{2}}\\
&=\frac{(\sqrt{A_{1}B_{1}-K^{2}})'(K-K^{2})-\sqrt{A_{1}B_{1}-K^{2}}(1-K)}{K^{2}+\left(\sqrt{A_{1}B_{1}-K^{2}}\right)^{2}},
\end{split}
\end{equation}
where $(\sqrt{A_{1}B_{1}-K^{2}})'$ denotes the derivative with respect to $K$. As $K<\sqrt{A_{1}B_{1}}<1$ and $(\sqrt{A_{1}B_{1}-K^{2}})'<0$, we have $r'(K)<0$ for $K\in (0,\sqrt{A_{1}B_{1}})$. 
And as $r(\sqrt{A_{1}B_{1}})=0$, this implies that for any $K\in (0,\sqrt{A_{1}B_{1}})$, $r(K)\geq0$ and in particular $x_{5}>1$. Thus $x_{4}>1$, and $\eta$ exists on $[0,1]$. This ends the proof of Proposition \ref{proplimit} in the case $K< \sqrt{A_{1}B_{1}}$. If $K> \sqrt{A_{1}B_{1}}$, then 
\begin{equation}
h(x)=\frac{\sqrt{K^{2}-A_{1}B_{1}}A_{1}\sinh(\sqrt{K^{2}-A_{1}B_{1}}x)}{(K^{2}-A_{1}B_{1})\cosh(\sqrt{K^{2}-A_{1}B_{1}}x)-K\sqrt{K^{2}-A_{1}B_{1}}\sinh(\sqrt{K^{2}-A_{1}B_{1}}x)},
\label{exprh2}
\end{equation}
and
\begin{equation}
x_{5}=\frac{\text{atanh}\left(\frac{\sqrt{K^{2}-A_{1}B_{1}}}{K}\right)}{\sqrt{K^{2}-A_{1}B_{1}}}.
\end{equation}
We define $\textcolor{black}{\phi} :X\rightarrow \text{atanh}\left(\frac{X}{K}\right)-X$, one has
\begin{equation}
\textcolor{black}{\phi}'(X)=\frac{K}{K^{2}-X^{2}}-1=\frac{X^{2}-(K^{2}-K)}{K^{2}-X^{2}}.
\end{equation}
This implies that if $K<1$, $\textcolor{black}{\phi}$ is increasing for $X\in [0,K)$ and if $K\geq 1$, $\textcolor{black}{\phi}$ is increasing for $X\in[\sqrt{K^{2}-K},K)$. As $\textcolor{black}{\phi}(0)=0$, we deduce that if $K<1$, as $\sqrt{K^{2}-A_{1}B_{1}}> 0$, $x_{5}>1$. If $K\geq 1$, then $A_{1}B_{1}<1<K$, thus $\sqrt{K^{2}-A_{1}B_{1}}>\sqrt{K^{2}-1}\geq\sqrt{K^{2}-K}$. Therefore,
\begin{equation}
x_{5}>\frac{\text{atanh}\left(\sqrt{1-\frac{1}{K^{2}}}\right)}{\sqrt{K^{2}-1}}.
\label{ineqx5}
\end{equation}
Now, let $r_{2}: K\rightarrow \text{atanh}\left(\sqrt{1-\frac{1}{K^{2}}}\right)-\sqrt{K^{2}-1}$. As previously for $r$ we have
\begin{equation}
r_{2}'(K)=\frac{-\sqrt{K^{2}-1}(K^{2}-K)-\frac{K}{\sqrt{K^{2}-1}}(K^{2}-K)}{K^{2}-(\sqrt{K^{2}-1})^{2}}\geq 0,\text{  for  }K>1.
\end{equation}
And $r_{2}(1)=0$, thus for any $K\geq 1$, $r_{2}(K)\geq 0$ and from \eqref{ineqx5} $x_{5}>1$. 
Finally if $K^{2}=A_{1}B_{1}$, then the expression \eqref{exprh2} does not hold anymore but $A_{1}B_{1}>0$ and \eqref{defh} becomes
\begin{equation}
\begin{split}
h'&=(\sqrt{A_{1}}+\sqrt{B_{1}}h)^{2},\\
h(0)&=0
\end{split}
\end{equation}
thus
\begin{equation}
(\sqrt{A_{1}}+\sqrt{B_{1}}h(x))=\frac{\sqrt{A_{1}}}{1-\sqrt{A_{1}B_{1}}x}.
\end{equation}
and $x_{5}=\sqrt{A_{1}B_{1}}^{-1}>1$. This ends the proof of 
Proposition \ref{proplimit}.\\
}
\label{limitlength}
\bibliographystyle{plain}
\bibliography{Biblio_ISS}
\end{document}